\pdfoutput=1
\RequirePackage[l2tabu, orthodox]{nag}

\documentclass[a4paper,reqno]{amsart}

\usepackage{todonotes}

\usepackage{fullpage}

\usepackage{lmodern}
\usepackage[T1]{fontenc}
\usepackage[utf8]{inputenc}
\usepackage[english]{babel}
\usepackage{microtype} 

\usepackage{amsmath,amssymb,mathrsfs,latexsym,mathtools,mathdots,enumerate,tikz,ytableau}
\usetikzlibrary{positioning}
\usepackage[enableskew,vcentermath]{youngtab}
\usepackage{graphicx}

\usepackage{hyperref}
\usepackage[capitalize]{cleveref}

\usepackage{bm}

\newtheorem{theorem}{Theorem}[section]
\newtheorem{lemma}[theorem]{Lemma}
\newtheorem{proposition}[theorem]{Proposition}
\newtheorem{definition}[theorem]{Definition}
\newtheorem{corollary}[theorem]{Corollary}
\newtheorem{example}[theorem]{Example}
\newtheorem{remark}[theorem]{Remark}
\newtheorem{conjecture}[theorem]{Conjecture}
\newtheorem{question}[theorem]{Question}

\newenvironment{yt}{\begin{array}{c}\begin{ytableau}}
    {\end{ytableau}\end{array}}

\newcommand{\defin}[1]{\emph{#1}}

\newcommand{\setC}{\mathbb{C}}

\newcommand{\avec}{\mathbf{a}}
\newcommand{\bvec}{\mathbf{b}}
\newcommand{\pvec}{\mathbf{p}}
\newcommand{\qvec}{\mathbf{q}}
\newcommand{\rvec}{\mathbf{r}}

\newcommand{\macw}{w}

\newcommand{\length}{\ell}

\newcommand{\domgeq}{\trianglerighteq}
\newcommand{\domleq}{\trianglelefteq}

\newcommand{\tabHook}{\mathcal{HT}}
\newcommand{\tabPerm}{\mathcal{PT}}
\newcommand{\CD}{\mathcal{CD}}
\newcommand{\stirling}[2]{ \genfrac\{\}{0pt}{}{#1}{#2}}

\newcommand{\perms}{\mathfrak{S}}
\newcommand{\PP}[1]{\perms^{(2)}_#1}

\newcommand{\U}{\mathcal{U}}

\def\QQ{\mathbb{Q}}

\newcommand{\schurS}{S}
\newcommand{\jackJ}{J}

\newcommand{\zonalZ}{Z}
\newcommand{\monomM}{M}
\newcommand{\psump}{\mathrm{p}} 
\newcommand{\F}{F} 
\def\hatK{\hat{K}} 
\def\sh{\star}
\def\ShSchur{\schurS^\sh}
\def\shJack{\jackJ^{\sh,(\a)}}
\def\ShZonal{Z^{\sh}}

\DeclareMathOperator{\lrmin}{lrmin}
\DeclareMathOperator{\id}{id}

\DeclareMathOperator{\Ch}{\vartheta}
\DeclareMathOperator{\Ko}{\mathfrak{K}}
\DeclareMathOperator{\type}{type}

\DeclareMathOperator{\QSym}{QSym}

\def\la{\lambda}
\def\a{\alpha}
\def\Sn{\mathfrak{S}}
\def\si{\sigma}
\def\eps{\varepsilon}

\def\Cha{\Ch^{(\a)}}
\def\Koa{\Ko^{(\a)}}

\newcommand{\pst}[1]{\widetilde{\mathrm{p}_{#1}^\star}}

\renewcommand{\phi}{\varphi}

\title[Shifted symmetric functions and multirectangular coordinates]
{Shifted symmetric functions and multirectangular coordinates of Young diagrams}
\author[P.~Alexandersson]{Per Alexandersson}
\author[V.~Féray]{Valentin Féray}

\keywords{shifted symmetric functions, Jack polynomials, multirectangular coordinates,
zonal spherical functions, characters of symmetric groups.}

\subjclass[2010]{Primary: 05E05. Secondary: 20C30.}

\begin{document}

\begin{abstract}
  In this paper, we study shifted Schur functions $S_\mu^\star$,
  as well as a new family of shifted symmetric functions $\Ko_\mu$ linked to Kostka numbers.
  We prove that both are polynomials in multi-rectangular coordinates,
  with nonnegative coefficients when written in terms of falling factorials.

  We then propose a conjectural generalization to the Jack setting.
  This conjecture is a lifting of Knop and Sahi's positivity result for usual Jack polynomials
  and resembles recent conjectures of Lassalle.
  We prove our conjecture for one-part partitions.
\end{abstract}

\maketitle


\section{Introduction}

We use standard notation for partitions and symmetric functions,
which is recalled in \cref{SectPreliminaries}.

\subsection{Shifted symmetric functions}
\label{SubsectIntroShiftedSym}

Informally, a shifted symmetric function is a formal power series in
infinitely many variables
$x_1$, $x_2$, \ldots
that has bounded degree and is symmetric in the ``shifted'' variables $x_1-1$, $x_2-2$, \ldots
(a formal definition is given in \cref{SubSecDefShSym}).

Many properties of symmetric functions have natural analogue
in the shifted framework. Unlike in symmetric function theory, it is often relevant
to evaluate a shifted symmetric function $\F$ on the parts 
of a Young diagram $\la=(\la_1,\dots,\la_\ell)$.
Then we denote $\F(\la) \coloneqq F(\la_1,\dots,\la_\ell,0,0,\dots)$.
It turns out that shifted symmetric functions are determined
by their image on Young diagrams, so that
the shifted symmetric function ring will be identified with a subalgebra
of the algebra of functions on the set of all Young diagrams
(without size nor length restriction).
\medskip

Shifted symmetric functions were introduced
by Okounkov and Olshanski in \cite{OkounkovOlshanskiShiftedSchur}.
In this paper, the authors are particularly interested in the basis
of \emph{shifted Schur functions}, which can be defined as follows:
for any integer partition $\mu$ and any $n \geq 1$,
\[
\ShSchur_\mu (x_1,\dotsc,x_n)
= \frac {\det \left( (x_i+n-i)_{\mu_j+n-j} \right)}
{\det \left( (x_i+n-i)_{n-j} \right)} ,
\]
where $(x)_k$ denotes the falling factorial $x(x-1) \dotsm (x-k+1)$.
Note the similarity with the definition of Schur functions \cite[p. 40]{Macdonald1995}:
in particular, the highest degree terms of $\ShSchur_\mu$
is the Schur function $\schurS_\mu$.
Shifted Schur functions are also closely related to \emph{factorial Schur polynomials},
originally defined by Biedenharn and Louck in \cite{BiedenharnLouckFactorialSchur}
and further studied, \emph{e.g.}, in \cite{MacdonaldVariationsSchur,MolevSaganLRFactorialSchur}.
These functions display beautiful properties:
\begin{itemize}
    \item Some well-known formulas involving Schur functions 
        have a natural extension to shifted Schur functions,
        \emph{e.g.}, the combinatorial expansion in terms of
        semi-standard Young tableaux \cite[Theorem 11.1]{OkounkovOlshanskiShiftedSchur}
        and the Jacobi-Trudi identity \cite[Theorem 13.1]{OkounkovOlshanskiShiftedSchur}.
    \item The evaluation $\ShSchur_\mu(\la)$ of a shifted Schur function indexed by $\mu$
        on a Young diagram $\la$ 
        has a combinatorial meaning:
        it vanishes if $\lambda$ does not contain $\mu$ and is related to the number of standard
        Young tableaux of skew shape $\lambda / \mu$ otherwise; see \cite[Theorem 8.1]{OkounkovOlshanskiShiftedSchur}.
        Note that this beautiful property has no analogue for usual (\emph{i.e.}, non-shifted) symmetric functions.
    \item Lastly, they appear as eigenvalues of elements of well-chosen bases in highest weight modules
        for classical Lie groups, see \cite{OkounkovOlshanski1998}.\medskip
\end{itemize}

There exists another connection between representation theory and shifted symmetric functions,
that we shall explain now.
It is well-know, see, \emph{e.g.}, \cite[Section I.7]{Macdonald1995} that irreducible character
values $\chi^\la_\tau$ of the symmetric group $\Sn_n$ are indexed by two partitions of $n$:
$\la$ stands for the irreducible representation we are considering and $\tau$ is the cycle-type
of the permutation on which we are computing the character.
Let us fix a partition $\mu$ of size $k$. 
For $n \ge k$, we can add parts of size $1$ to $\mu$ to get a partition of size $n$,
that we will denote $\mu 1^{n-k}$.
Then consider the following function on Young diagrams:
\begin{equation}
\Ch_\mu(\la)= 
\begin{cases}
    |\lambda|(|\lambda|-1)\cdots (|\lambda|-|\mu|+1) \frac{\chi^\lambda_{\mu1^{|\lambda|-|\mu|}}}
    {\chi^\lambda_{1^{|\mu|} }} &\text{ if }|\lambda| \ge |\mu|;\\
    0 &\text{ if }|\lambda| < |\mu|.
\end{cases}
\label{EqDefCh1}
\end{equation}
It turns out that $\Ch_\mu$ is a shifted symmetric function 
\cite[Proposition 3]{KerovOlshanskiPolFunc}.
The family of \emph{normalized characters} $\Ch_\mu$ is an important
tool in asymptotic representation theory: in particular they are
central objects in the description by Ivanov, Kerov and Olshanski
of the fluctuations of random Young diagrams distributed with Plancherel
measures \cite{IvanovOlshanski2002}.
Finding combinatorial or analytic expressions for $\Ch_\mu(\la)$ in terms
of various set of coordinates of the Young diagrams
has been the goal of many research papers 
\cite{Biane2003,GouldenRattan2007,SniadyGenusExpansion,RattanStanleyTopDegree,
FerayPreuveStanley,FerayPreuveKerov,NousKerovExplicitInterpretation,PetrulloSenatoKerov}.
\bigskip

In this paper we will consider a third family
of shifted symmetric functions.
Following Macdonald \cite[page 73]{Macdonald1995},
we denote $K^{\la}_{\tau}$ the number of semi-standard Young tableaux
of shape $\la$ and type $\tau$, often called Kostka number.
Equivalently, $(K^{\la}_{\tau})_{\tau \vdash |\lambda|}$ is the family of coefficients
of the monomial expansion of Schur functions
\begin{equation}
\schurS_\la = \sum_{\tau \vdash |\lambda|} K^{\la}_{\tau} \monomM_\tau.
\label{EqSchurInMon}
\end{equation}
By analogy with \cref{EqDefCh1} (recall that irreducible character values of $\Sn_n$
are roughly the coefficients of the power-sum expansion of Schur polynomial),
we define
\begin{equation}\label{EqDefKo1}
\Ko_\mu(\lambda) =
\begin{cases}
n(n-1)\dotsm (n-k+1) \dfrac{  K^\lambda_{{\mu1^{n-k}} } }{ K^\lambda_{1^n}   } \text{ if } n \geq k \\
0 \text{ otherwise}.
\end{cases}
\end{equation}
It is easy to show that $\Ko_\mu$ is also a shifted symmetric function
(in fact it is easy to express it in terms of $\Ch_\mu$, see \cref{PropKoaOnCha}).
Note that, for any diagram $\la$, the quantity $\Ko_\mu(\la)$ is non-negative.
We will be interested in non-negativity properties 
of the function $\Ko_\mu$ itself (and not only its specializations).

\subsection{Multirectangular coordinates}
In a beautiful paper \cite{Stanley2003}, Stanley proved that the function
$\Ch_\mu$ evaluated on a rectangular Young diagram $\la=(r^p)=(r,\dots,r)$ (with $p$ parts),
has a very nice expression
\begin{equation}
    \Ch_\mu \big( (r^p) \big)
= (-1)^{|\mu|} \sum_{\substack{ \si, \tau \in \Sn_{|\mu|} \\ \si \, \tau = \pi }} p^{|C(\si)|} (-r)^{|C(\tau)|},
\label{EqChRectangle}
\end{equation}
where $\Sn_{|\mu|}$ is the symmetric group of size $|\mu|$,
$\pi$ is a fixed permutation of cycle-type $\mu$
and $C(\si)$ and $C(\tau)$ the set of cycles of $\si$ and $\tau$, respectively.

In order to generalize this formula, Stanley introduced the notion
of multirectangular coordinates\footnote{Beware that the definition
of multirectangular coordinates that we use here is slightly different than Stanley's.
To recover Stanley's from ours, just set $q_i=r_i+\dots+r_d$.}:
to two lists of non-negative integers $\pvec=(p_1,\dots,p_d)$ and
$\rvec=(r_1,\dots,r_d)$, we associate the Young diagram drawn in \cref{FigDiagMulti}.
We denote it $\rvec^\pvec$.

\begin{figure}[t]
    \centering
	\begin{tikzpicture}[scale=0.35,y=-1cm]
	\draw[black] (0, 0)--(0,6);
	\draw[black] (2, 0)--(2,6);
	\draw[black] (6, 0)--(6,4);
	\draw[black] (8,0) --(8,2);
	\draw[<->]  (-0.5, 0)--(-0.5, 2);
	\draw[<->]  (-0.5, 2)--(-0.5, 4);
	\draw[<->]  (-0.5, 4)--(-0.5, 6);
	\draw[black] (0, 0)--(8,0);
	\draw[black] (0, 2)--(8,2);
	\draw[black] (0, 4)--(6,4);
	\draw[black] (0, 6)--(2,6);
	\draw[<->] (0 , -0.5)--(2, -0.5);
	\draw[<->] (2 , -0.5)--(6, -0.5);
	\draw[<->] (6 , -0.5)--(8, -0.5);
	\node at (-1 , 1.0) {$p_1$};
	\node at (-1 , 3.0) {$p_2$};
	\node at (-1 , 5.0) {$p_3$};
	\node at ( 1 , -1) {$r_3$};
	\node at ( 4 , -1) {$r_2$};
	\node at ( 7 , -1) {$r_1$};
	\end{tikzpicture}
    \caption{Multirectangular coordinates of a Young diagram}
    \label{FigDiagMulti}
\end{figure}
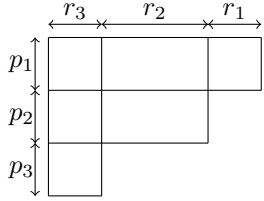

Throughout the article, we will consider $d$ as a fixed positive integer,
so that we consider diagrams which are superpositions of a given number of rectangles.
However, any diagram has such a description for some value of $d$,
thus our formulas, \emph{e.g.}, \cref{EqShSchurN,EqKoN},
give in effect  the evaluation of shifted symmetric functions on any diagram.

Stanley conjectured the following positivity property,
proved later in \cite{FerayPreuveStanley} 
(a simpler proof has been given shortly after in \cite{NousBoundsOnCharacters}).

\begin{theorem}[\cite{FerayPreuveStanley}]
    For each $i$ between $1$ and $d$, set $q_i=r_i+\dotsb+r_d$.
    Then, for every partition $\mu$, the quantity $(-1)^{|\mu|} \Ch_\mu(\rvec^\pvec)$ is a polynomial with non-negative integer
    coefficients in the variables $p_1, \dotsc, p_d$, $-q_1, \dotsc, -q_d$.
    \label{ThmPositivityCh1}
\end{theorem}

Polynomiality in the statement above is an easy consequence of 
the shifted symmetry of $\Ch_\mu$; see \cref{CorolShiftSymPoly}.
Integrality of coefficients is a bit harder,
but was established by Stanley in the paper where he stated the conjecture.
The most interesting part is the non-negativity,
which is established 
by finding a combinatorial interpretation,
which was also conjectured by Stanley \cite{Stanley-preprint2006}.
This combinatorial formula is presented later in this paper in \cref{EqChN}.
\bigskip

We now consider the following question:
do the expressions of other families of shifted symmetric functions,
namely $(\ShSchur_\mu)$ and $(\Ko_\mu)$, in terms of multirectangular coordinates
also display some positivity property?
We know that, for any diagram $\lambda$,
both quantities $\ShSchur_\mu(\lambda)$ and $\Ko_\mu(\lambda)$ are non-negative
(for the latter, it is obvious from the definition; 
for the former, see \cite[Theorem 8.1]{OkounkovOlshanskiShiftedSchur}).
In other terms, when we specialize $p_1$, \dots, $p_d$, $r_1$, \dots, $r_d$
to non-negative integer values, then the polynomial expressions
$\ShSchur_\mu(\rvec^\pvec)$ and $\Ko_\mu(\rvec^\pvec)$ specialize to non-negative values.

A first guess would be that these polynomials have non-negative coefficients,
but this is not the case.
A natural basis of the polynomial ring $\QQ[\pvec,\rvec]$ (other than monomials),
whose elements 
have the non-negative specialization property described above,
is the falling factorial basis
\[
\bigg( (p_1)_{a_1} \cdots (p_d)_{a_d} (r_1)_{b_1} \cdots (r_d)_{b_d} \bigg)_{a_1,\dots,a_d,b_1,\dots,b_d \ge 0}.
\]
Our first main result is the following:
\begin{theorem}[First main result]
    \label{thm:mainshifted}
For every partition $\mu$,
the polynomials  $\ShSchur_\mu(\rvec^\pvec)$ and $\Ko_\mu(\rvec^\pvec)$
have non-negative rational coefficients in the falling factorial basis.
\end{theorem}
This Theorem is proved in \cref{SectAlpha1}.
\smallskip

We end this section by a short discussion on the motivations
beyond multirectangular coordinates.
\begin{itemize}
  \item Take a diagram $\lambda$ and replace each box by a square
    of $s \times s$ boxes. The resulting diagram is denoted $s\cdot \la$
    and is called a {\em dilation} of $\la$.
    This results in multiplying all multirectangular coordinates by $s$.
    This makes multirectangular coordinates suited to study asymptotics
    of functions $f(\la)$ on Young diagrams in the {\em balanced case}
    ({\em i.e.} when the number of rows and columns of $\la$ both grow as $\sqrt{n}$);
    see \cite{NousBoundsOnCharacters}.
  \item The polynomials obtained by writing shifted symmetric functions
    in terms of multirectangular coordinates display some interesting symmetry structure.
    It is easy to see that they are {\em diagonally quasisymmetric} in $\pvec$ and $\rvec$.
    A finer analysis shows that the algebra of functions on Young diagrams that are polynomial
    in multirectangular coordinates turns out to be isomorphic to the quasisymmetric
    function ring $\QSym$, see \cite{FerayEtJeans}.
  \item As shown above (\cref{ThmPositivityCh1}),
    nice combinatorial formulas for normalized character
    values in terms of multirectangular coordinates
    have been found.
    This formula makes a new connection between
    symmetric group characters 
    and maps combinatorics.
    Thanks to this connection, characters on permutations
    of type $(k,1^{n-k})$ in terms of multirectangular coordinates
    can be computed efficiently \cite[Section 4]{NousUnicellulaire}.
  \item A similar formula have been found for the so-called
    zonal spherical functions of the Gelfand pair $(\Sn_{2k},H_k)$
    ($H_k$ is the hyperoctahedral group; see \cref{ssec:zonal_bakground}
    for a short account on these objects).
    A conjecture has been made by Lassalle \cite{LassalleConjecturePQ} to extend this to 
    {\em Jack characters}; we will discuss it in the next Section.
    A motivation for the current work is 
    an attempt to better understand this conjecture.
  \item Finally note that, if we set $p_i=1$ for each $i$,
    then we have $r_i=\la_i-\la_{i+1}$, so that
    multirectangular coordinates contain a very simple transformation
    of the parts of the Young diagram.
    Therefore \cref{thm:mainshifted} contains the following statement 
    in terms of the parts of the Young diagrams:
    the functions $\ShSchur_\mu(\la)$ and $\Ko_\mu(\la)$
    restricted to diagrams $\la$ with at most $\ell$ parts
    are polynomials
    in $\la_1,\dots,\la_\ell$ with nonnegative coefficients
    in the basis
    \[
    \bigg( (\la_1-\la_2)_{b_1} \cdots (\la_{\ell-1}-\la_\ell)_{b_\ell} (\la_\ell)_{b_\ell} \bigg)_{b_1,\dots,b_\ell \ge 0}.
    \]
    Results and conjectures of the next sections also
    have easy (conjectural) consequences in terms of parts of the partition,
    that we shall not write down.
\end{itemize}

\subsection{Jack analogues}
\label{SubsectIntroJack}
In a seminal paper \cite{Jack1970/1971},
Jack introduced a family of symmetric functions
$J^{(\a)}_\mu$ depending on an additional parameter
$\a$. These functions are now called \emph{Jack polynomials}.
They have been extensively studied from an algebraic combinatorics point of view;
see, \emph{e.g.}, \cite{Stanley1989}, \cite[Section VI.10]{Macdonald1995} and \cite{KnopSahiCombinatoricsJack}.
For $\a=1$, Jack polynomials coincide with Schur polynomials (up to multiplication by a scalar).
On the other hand, they are degenerate cases of Macdonald polynomials, 
when both parameters $q$ and $t$ tend to $1$ with $q=t^\a$.

We will also consider $\a$-shifted symmetric functions,
which are formal power series in infinitely many variables 
$x_1$, $x_2$, \ldots
that have bounded degree and are symmetric in $x_1-1/\a$, $x_2-2/\a$, \ldots
(a formal definition is given in \cref{SubSecDefShSym}).
While there is a trivial isomorphism between $1$-shifted symmetric functions
and $\a$-shifted symmetric functions 
\cite[Remark I.7]{OkounkovOlshanskiShiftedSchur},
when dealing with Jack polynomials,
it is more convenient to work with the $\a$-version.

In this setting,
Okounkov and Olshanski have defined and studied \emph{shifted Jack polynomials},
that we shall denote here by $\jackJ^{\sh,(\a)}_\mu$.
\begin{itemize}
    \item On the one hand, the function $\jackJ^{\sh,(1)}_\mu$
is a multiple of the shifted Schur function $\schurS^\sh_\mu$.
\item On the other hand, the top degree component of $\jackJ^{\sh,(\a)}_\mu$
    is the usual Jack polynomial $\jackJ_\mu$.
\end{itemize}
Besides, they admit a combinatorial description in terms of tableaux
\cite[Equation (2.4)]{OkounkovOlshanskiShiftedJack},
can be characterized by nice vanishing conditions
\cite{KnopSahiCombinatoricsJack}
and
appear in some binomial formulas for Jack polynomials \cite{LassalleBinomiauxGeneralises}.
All these properties make them natural extensions of shifted Schur functions,
that are worth being investigated.
\bigskip

Another family of $\a$-shifted symmetric functions,
which is a natural extension of $\Ch_\mu$,
has been recently introduced and studied by Lassalle \cite{LassalleConjecturePQ,LassalleFreeCumulants}.
Expanding Jack polynomials $J_\la^{(\a)}$ 
in power-sum symmetric function basis, 
we define the coefficients 
$\theta_{\tau}^{(\a)}(\la)$ by:
\begin{equation}
\label{eq:jack-characters}
J_\la^{(\a)}=\sum_{\substack{\tau: \\
|\tau|=|\la|}} 
\theta_{\tau}^{(\a)}(\la)\ \psump_{\tau}.
\end{equation}
Then, for a fixed partition $\mu$, we consider the following function on Young diagrams:
\begin{equation}
\label{EqDefChA}
\Ch_{\mu}^{(\a)}(\la)=
\begin{cases}
\binom{|\la|-|\mu|+m_1(\mu)}{m_1(\mu)}
\ z_\mu \ \theta^{(\a)}_{\mu,1^{|\la|-|\mu|}}(\la)
&\text{if }|\la| \ge |\mu| ;\\
0 & \text{if }|\la| < |\mu|.
\end{cases}
\end{equation}
Here, $z_\mu$ denotes the standard quantity $1^{m_1}m_1!2^{m_2}m_2!\dotsm$ if $m_i=m_i(\mu)$ is the number of parts
of $\mu$ of size $i$. When $\a=1$, 
the function $\Ch_{\mu}^{(1)}$ coincide with $\Ch_\mu$ \cite[Eq. (1.1)]{LassalleConjecturePQ}.
The $\a$-shifted symmetry of $\Ch_{\mu}^{(\a)}$ is non-trivial,
see \cite[Proposition 2]{LassalleConjecturePQ}.
\bigskip

The family $\Ko_\mu$ that we introduced in \cref{SubsectIntroShiftedSym} also has a natural
analogue for a general parameter $\a$.
Consider the monomial expansion of Jack polynomials:
\[\jackJ^{(\a)}_\la = \sum_{\tau \vdash |\lambda|} \hatK^{\la,(\a)}_{\tau} \monomM_\tau.\]
Then define by analogy with \cref{EqDefKo1}
(recall that $\hatK^{\la,(\a)}_{(1^{|\la|})}=n!$ for all partitions $\la$)
\begin{equation}\label{EqDefKoA}
    \Ko^{(\a)}_\mu(\lambda) =
\begin{cases}
    \frac{1}{(n-k)!} \, \hatK^{\la,(\a)}_{{\mu1^{n-k}} } & \text{ if } n \geq k \\
0 &\text{ otherwise.}
\end{cases}
\end{equation}
We will see in \cref{SubsectDefJack} that $\Ko^{(1)}_\mu=\Ko_\mu$.
(As for $\Ch_{\mu}^{(1)}=\Ch_\mu$, there is nothing deep or surprising in this specialization,
only checking that the normalization factors coincide needs a bit of care.)
\bigskip

As in the case $\a=1$, shifted symmetry implies a polynomial dependence
in multirectangular coordinates (see \cref{CorolShiftSymPoly};
the coefficients here are \emph{a priori}
rational functions in $\a$)
and one can investigate this expression.
In this direction, M.~Lassalle has formulated a
conjecture generalizing \cref{ThmPositivityCh1}.
\begin{conjecture}[\cite{LassalleConjecturePQ}]
    For each $i$ between $1$ and $d$, set $q_i=r_i+\dots+r_d$.
    Then, for every partition $\mu$, the quantity $(-1)^{|\mu|} \Ch^{(\a)}_\mu(\rvec^\pvec)$ is a polynomial with non-negative integer
    coefficients in the variables $\alpha-1$, $p_1$, \dots, $p_d$, $-q_1$, \dots, $-q_d$.
    \label{ConjLassallePQ}
\end{conjecture}
We also formulate a conjecture for general $\a$.
Consider the ring of polynomials in variables $\a$, $p_1,\dots,p_d$, $r_1,\dots,r_d$
and let the following basis
\[
\bigg( \a^c (p_1)_{a_1} \cdots (p_d)_{a_d} (r_1)_{b_1} \cdots (r_d)_{b_d} \bigg)_{c,a_1,\dots,a_d,b_1,\dots,b_d \geq 0}.
\]
be the \emph{$\a$-falling factorial basis} of this ring.
\begin{conjecture} \label{ConjMain}
For every partition $\mu$,
the quantities  $\a^{|\mu|-\mu_1} \, \shJack_\mu(\rvec^\pvec)$ and $\Ko^{(\a)}_\mu(\rvec^\pvec)$
are polynomials with non-negative rational
coefficients in the $\a$-falling factorial basis.
\end{conjecture}
This conjecture implies in particular that,
for each diagram $\la$, $\Ko^{(\a)}_\mu(\lambda)$ is a polynomial with non-negative coefficients in $\alpha$.
This weaker statement was conjectured by Stanley and Macdonald around 1990
and proved a few years later by Knop and Sahi \cite{KnopSahiCombinatoricsJack}
(in fact, they prove also that the coefficients are integers, which we do not discuss here).

\cref{ConjMain} has been tested numerically for partition up to size $9$ and $d=4$.
We explain briefly in \cref{SectComputer} how we computed the expression of $\shJack_\mu(\rvec^\pvec)$ 
and $\Ko^{(\a)}_\mu(\rvec^\pvec)$ in terms of multirectangular coordinates.


In addition to the special value $\a=1$,
we are able to prove another particular case of the conjecture above, corresponding to one-part partitions $\mu=(k)$.
As observed in \cref{Eq:ShJKoa_OnePart} below, we have 
\hbox{$\shJack_{(k)}= k! \Ko^{(\a)}_{(k)}$},
so that both parts of the conjecture coincide in this case.
\begin{theorem}
    [Second main result]
    For any integer $k$,
    the quantity $\shJack_{(k)}(\rvec^\pvec)=k! \Ko^{(\a)}_{(k)}(\rvec^\pvec)$
    is a polynomial with non-negative rational coefficients in the $\a$-falling factorial basis.
    \label{ThmOnePart}
\end{theorem}
This theorem is proved in \cref{SubsectKkFFPos}.
A key step is a new combinatorial description of $\Ko^{(\a)}_{(k)}$
that could be interesting in itself; see \cref{thm:NewCombFormula2}.

\subsection{Methods}
For our first main result (\cref{thm:mainshifted}),
the strategy of the proof is the following.
\begin{itemize}
    \item We use known combinatorial formulas for $\Ch^{(1)}_\mu$ in terms of rectangular coordinates.
        From them, we can easily deduce similarly looking formulas for $\ShSchur_\mu(\rvec^\pvec)$ and $\Ko_\mu(\rvec^\pvec)$
        (\cref{combFormulasInFerayN}).
    \item A classical trick consisting in considering partitioned objects allows to rewrite these quantities in the
      falling factorial basis (\cref{LemAFF}).
    \item We then use some representation-theoretical manipulation to prove the positivity (\cref{lem:nonnegsum}).
\end{itemize}
\bigskip

The method of proof for the second main result is completely different.
Indeed, for general $\a$, there is no known combinatorial formula for $\Ch^{(\a)}_\mu$.
In this case, our starting point is the Knop and Sahi combinatorial formula
for Jack polynomials in the monomial basis.
Reinterpreted with our point of view, this result gives a combinatorial description
of $\Ko_\mu^{(\a)}(\la)$.
Unfortunately, as is, this result cannot be used to compute expressions
in multirectangular coordinates.
But in the special case $\mu=(k)$, we were able to construct a bijection
that yields a new combinatorial expression for $\Ko_{(k)}^{(\a)}(\la)$,
which is suitable for evaluation in multirectangular coordinates,
(\cref{thm:NewCombFormula2}).

Whether this construction has a natural extension to any partition $\mu$
is an open problem.
We have not been able to find one but, somehow, \cref{ConjMain}
suggests that it might exist.

\subsection{Discussion}
One of the original motivations of this paper was to unify two seemingly different approaches on Jack polynomials:
\begin{itemize}
    \item The classical approach, initiated by Stanley \cite{Stanley1989} and Macdonald \cite[Section VI,10]{Macdonald1995},
        consists in finding a formula for $J^{(\a)}_\la$ for a fixed partition $\la$,
        as a weighted sum of combinatorial objects.
        A seminal result in this approach is the already mentioned formula of Knop and Sahi \cite{KnopSahiCombinatoricsJack},
        which implies that the coefficients of (augmented) monomial symmetric functions in $J^{(\a)}_\la$
        are polynomials in $\a$ with non-negative integer coefficients.
    \item The second approach is sometimes referred to as \emph{dual}.
        One looks at the coefficient of a fixed power-sum in Jack symmetric functions
        as a function of the partition $\la$
        which indexes the Jack function.
        This is how $\Ch^{(\a)}_\mu$ is defined.
        Then one expresses this function in terms of some set of coordinates of Young diagrams, 
        \emph{e.g.}, multirectangular coordinates.
        In this approach, most positivity questions are still open: \cref{ConjLassallePQ} is an instance of such open questions,
        see \cite{LassalleFreeCumulants,DFSOrientability} for other examples.
\end{itemize}
The hope was to make a link between the two approaches to be able to solve conjectures in the dual approach,
using the results from the classical approach.
We did not achieve this goal but our work could be a first step to bring together both approaches. Indeed,
\begin{itemize}
    \item our main conjecture (\cref{ConjMain}) involves multirectangular coordinates,
but implies the Knop--Sahi positivity result;
\item tools that we use to establish our partial results come both from the dual approach 
    (for our first main result) and the classical one (for the second main result).
\end{itemize}

To finish the discussion section, let us mention that Knop and Sahi have proposed
a different positivity conjecture on shifted Jack polynomials \cite[Section 7]{KnopSahiShiftedSym}.
It seems to be unrelated to ours.

We also point out that shifted Jack polynomials have some Macdonald analogues; see \cite{Okounkov1998}.
On the other hand, we do not know a Macdonald analogue of the family $\Cha_\mu$:
the coefficients of the power-sum expansion of Macdonald polynomials
are not $q,t$ shifted symmetric functions (even after appropriate normalization).
Since $\Cha_\mu$ and its shifted symmetry
play an important role in our work, we did not consider the Macdonald setting.

\subsection{The case  \texorpdfstring{$\a=2$}{a=2}}

For $\a=2$, Jack polynomials are known to specialize to the so-called zonal polynomials.
The latter appear in the theory of Gelfand pairs, see \emph{e.g.}, \cite[Section 7]{Macdonald1995}.
We denote by $\ShZonal_\mu \coloneqq \jackJ^{\sh,(2)}_\mu$ the shifted zonal polynomial,
\emph{i.e.}, the shifted Jack polynomial for $\a=2$.

In \cref{SectAlpha2}, we give some new formulas for $\ShZonal_\mu$ and $\Ko^{(2)}_\mu$, 
similar to the case $\a=1$.
Unfortunately, we have not been able to use them to prove \cref{ConjMain} for $\a=2$.
This case remains open.
%
%

\subsection{Outline of the paper}
\cref{SectPreliminaries} gives the necessary notation and background.
In \cref{SectAlpha1}, we prove our first main result: the positivity of shifted Schur functions
in the falling factorial basis.
\cref{SectAlpha2} gives some analogous formulas for $\a=2$ (although we cannot prove the positivity in this case).
In \cref{SectOnePart}, we prove our second main result: the positivity for one-part partitions $\mu=(k)$,
using a new combinatorial interpretation of $\Koa_{(k)}(\la)$.
We conclude the paper (\cref{SectComputer})
by a short description of the computer tests supporting our main conjecture.


\section{Preliminaries}
\label{SectPreliminaries}
\subsection{Partitions, Young diagrams and hooks}
\label{SubsecPartitions}

We review basic notions in the theory of Young tableaux and symmetric functions.
This material can be found in standard reference literature such as \cite{Macdonald1995}.

A \emph{partition} $\lambda = (\lambda_1,\dotsc,\lambda_n)$ is a finite weakly decreasing sequence of non-negative integers.
The number of positive entries is the \emph{length} of the partition, denoted $\length(\lambda)$
and the size, $|\lambda|$, is the sum of all entries in $\lambda$.
The number of entries in $\lambda$ which are equal to $j$ is denoted $m_j(\lambda)$.

The partition with $k$ entries equal to $1$ is denoted $1^k$.
We use the standard convention that $\lambda_i=0$ if $i>\length(\lambda)$.

We say that $\lambda$ \emph{dominates} $\mu$ if $\lambda_1+\lambda_2 + \dots + \lambda_j \geq \mu_1+\mu_2 + \dots + \mu_j$
for all $j=1,2,\dots$. This defines a partial order on the set of partitions of equal size
and this relation is denoted $\lambda \domgeq \mu$.
Another partial order on partitions of the same size is the following:
 we say that $\mu$ \defin{refines} $\la$ if there exists an ordered set-partition $I_1,\dots,I_{\ell(\la)}$ of 
  the set $\{1,\dots,\ell(\mu)\}$ such that, for each $i \le \ell(\la)$, one has $\la_i=\sum_{j \in I_i} \mu_j$.
  If $\mu$ refines $\la$, then $\la$ dominates $\mu$, but the converse is not true.

We also write $\lambda \supseteq \mu$ if $\lambda_i \geq \mu_i$ for all $i$.
 Note that this last partial order compares partition with different sizes.
\medskip 

To every partition, we associate a \emph{Young diagram}, which is a left-justified arrangement of boxes in the plane,
where the number of boxes in row $i$ (from the top) is given by $\lambda_i$. See \cref{fig:diagramExample} for an example.
 Partitions will be often identified with their Young diagram.
Whenever $\lambda \supseteq \mu$, we define the skew Young diagram of shape $\lambda/\mu$
as the diagram obtained from $\lambda$ by removing
the of the diagram $\mu$.

\begin{figure}[h]
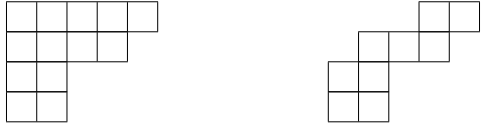

 $\young(\hfil\hfil\hfil\hfil\hfil,\hfil\hfil\hfil\hfil,\hfil\hfil,\hfil\hfil)$
\hspace{2cm}
$\young(:::\hfil\hfil,:\hfil\hfil\hfil,\hfil\hfil,\hfil\hfil)$
  \caption{Diagram of shape $(5,4,2,2)$ and $(5,4,2,2)/(3,1)$.}\label{fig:diagramExample}
\end{figure}


The \emph{arm-length} $a_\lambda(s)$ of a box $s$ in the Young diagram $\lambda$ is the number of boxes to the right of $s$,
and the \emph{leg-length} $l_\lambda(s)$ is the number of boxes below $s$.
The \emph{hook} value of a box is given by the arm-length, plus the leg length plus one. The hook values
  for a Young diagram is given in \cref{fig:hookLengthExample}.

\begin{figure}[!ht]
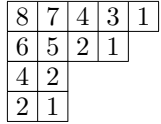

\centering
  $\young(87431,6521,42,21)$
 \caption{Diagram with hook lengths.}\label{fig:hookLengthExample}
\end{figure}
The \emph{hook-product} of a Young diagram is the product of all hook values of the boxes in the diagram.
There are two $\alpha$ deformations of the hook-product, given by
\[
H_\lambda^{(\alpha)} = \prod_{s \in \lambda} (\alpha a_\lambda(s) + l_\lambda(s) + 1), \quad
{H'}_\lambda^{(\alpha)} = \prod_{s \in \lambda} (\alpha a_\lambda(s) + l_\lambda(s) + \alpha).
\]

\bigskip

A \defin{semi-standard Young tableau} of shape $\lambda$
is a filling of a Young diagram of shape $\lambda$ with positive integers, such that
each row is weakly increasing left to right, and each column is strictly increasing from top to bottom.
We define skew semi-standard Young tableaux in the same way, but for diagrams of shew shape $\lambda/\mu$.

The \defin{type} $\tau=(\tau_1,\tau_2,\dots,\tau_l)$ of a semi-standard Young tableau $T$
is the integer composition $\tau$ such that $\tau_i$ counts the number of boxes in $T$ filled with $i$.
A Young tableau is \defin{standard} if its type is $1^n$ where $n$ is the number of boxes in its diagram.
The \emph{Kostka coefficient} $K^\lambda_\tau$ is the number of semi-standard Young tableaux of shape $\lambda$
and type $\tau$.

\subsection{Symmetric functions}

In this section, we briefly present some bases of the symmetric function ring and some relations between them.

The \defin{monomial symmetric functions}, denoted $M_\mu$ and indexed by partitions, are defined as
\[
 M_\mu = \sum_{ \rho \text{ distinct permutations of } \mu} x^{\rho}.
\]
The \defin{power sum symmetric functions} $\psump_\mu$ are defined as
\[
 \psump_{\mu} = \psump_{\mu_1}\psump_{\mu_2}\dotsm \psump_{\mu_l}, \text{ where } \psump_j = x_1^j + x_2^j + \dotsb.
\]
Finally, as mentioned in the introduction, the Schur function $S_\mu$ is defined by 
its restrictions to finite number of variables:
\[
\schurS_\mu (x_1,\dotsc,x_n)
= \frac {\det \left( x_i^{\mu_j+n-j} \right)}
{\det \left( x_i^{n-j} \right)} .
\]
They can be alternatively defined in terms of tableaux,
by their expansion in the monomial basis given in \cref{EqSchurInMon}.
\medskip

Let us consider the monomial expansion of power sum symmetric functions.
For a partition $\nu$ of $k$, there exists a collection of numbers 
$(L_{\nu,\mu})_{\mu \vdash k}$ such that:
\begin{equation}
    \psump_\nu = \sum_{\mu \vdash k} L_{\nu,\mu} M_\mu.
    \label{EqPowerOnMon}
\end{equation}
We record two important properties of these coefficients.
\begin{description}
\item[triangularity]
$L_{\nu,\mu} = 0$ unless $\nu$ refines $\mu$. In particular, $L_{\nu,\mu} = 0$ unless $\nu \domleq \mu$; \cite[p. 103]{Macdonald1995}.

\item[stability]
$L_{\nu \cup (1),\mu \cup (1)} = (m_1(\nu)+1) \cdot L_{\nu,\mu}$.
This follows easily from the combinatorial interpretation in \cite[Prop. 7.7.1]{StanleyEC2}.
In particular iterating it, one has that
        \begin{equation}
            L_{\nu 1^{n-k}, \mu 1^{n-k}} = (m_1(\nu)+1) \cdots (m_1(\nu)+n-k)\, L_{\nu,\mu}
            = \frac{(m_1(\nu)+n-k)!}{m_1(\nu)!}\, L_{\nu,\mu}.
            \label{EqLAddNmKUn}
        \end{equation}
\end{description}

\begin{lemma}
Let $\pi$ be a permutation and $\type(\pi)$ its cycle-type.
Then $L_{\type(\pi),\mu}$ is the number of functions 
$f:[k] \to [\ell(\mu)]$ that are constant on cycles of $\pi$
and such that the size of the pre-image $f^{-1}(i)$ is exactly $\mu_i$.
\label{LemLPerm}
\end{lemma}
\begin{proof}
    Note that a function on $[k]$ that is constant on cycles of $\pi$
    can be equivalently seen as a function on $C(\pi)$, the cycles of $\pi$.
    Using this, the lemma is just a rewording of \cite[Prop. 7.7.1]{StanleyEC2}.
\end{proof}
 
\emph{Note:} since this article deals mostly with \emph{shifted symmetric functions},
we will use the terminology \emph{usual symmetric functions}
to refer to symmetric functions, as described in this Section.
\medskip

\emph{Notation:}
Beware that the letter $p$ is used at the same time for power-sum symmetric functions
and multi-rectangular coordinates.
Since both uses of $p$ are classical, we decided not to change the notation.
We hope that it will not create any difficulty.
Note that we use a slightly different font $\psump$ for power-sums,
which may help in case of doubts.

 \subsection{Characters of the symmetric group and symmetric functions}
  \label{Subsec:Frobenius}
Irreducible representations of the symmetric group $S_n$ 
are indexed by partitions $\la$ of $n$.
If $\pi$ lies in $S_n$, we will denote 
$\chi^\la(\pi)$ the \emph{character} of the representation
indexed by $\la$, evaluated on the permutation $\pi$.

This character depends only on the cycle-type $\tau$ of $\pi$.
Therefore, if $\la$ and $\tau$ are two partitions of $n$,
we will also denote by $\chi^\la_\tau$ the character
$\chi^\la(\pi)$, where $\pi$ is \emph{any} permutation of cycle-type $\tau$.
\medskip

By a result of Frobenius, these irreducible character values
appear in symmetric function theory; see, \emph{e.g.}, \cite{Macdonald1995,SaganSymmetric}.
Namely, for any partition $\la$ of $n$,
\begin{equation}
    \schurS_\lambda = \sum_\tau \chi^\lambda_\tau \frac{\psump_\tau}{z_\tau}.
    \label{EqFrobenius}
\end{equation}

\subsection{Jack polynomials}

\label{SubsectDefJack}
In this subsection, we review a few properties of Jack polynomials
that are useful in this paper.
For details and proofs, we refer to Macdonald's seminal book \cite[Section VI,10]{Macdonald1995}.
We also use the notation of this book:
in particular we work with the $J$-normalization of Jack polynomials.

Denote $\QQ(\a)$ the field of rational fraction in an indeterminate $\a$
over the rational numbers.
We consider the ring $\Lambda_{\QQ(\a)}$ of symmetric functions
over the field $\QQ(\a)$.
Then the family of Jack polynomials
$(\jackJ_\la^{(\a)})$, indexed by partitions $\la$,
is a basis of $\Lambda_{\QQ(\a)}$.
Recall that $\hatK^{\la,(\a)}_{\tau}$ denotes the coefficients of the monomial expansion
 of $\jackJ^{(\a)}_\la$, { i.e.}
  $\jackJ^{(\a)}_\la = \sum_{\tau \vdash |\lambda|} \hatK^{\la,(\a)}_{\tau} \monomM_\tau$.
   \medskip

For $\a=1$, Jack polynomials coincide up to a multiplicative constant, with Schur functions,
namely:
\[ J_\la^{(1)} = H^{(1)}_\la S_\la,\]
where $H^{(1)}_\la=|\la|!/\chi^\lambda_{1^{|\la|}}$ is the hook product of $\la$.
Taking monomial coefficients, we get that, for any $\la$ and $\tau$,
one has $\hatK^{\la,(1)}_\tau = H^{(1)}_\la K^\la_\tau$.
Using also the fact that $K^\la_{1^n}=n!/H^{(1)}_\la$ 
(this is the number of standard Young tableaux of shape $\la$),
we have
\[\Ko_\mu^{(1)}(\la) = \frac{1}{(n-k)!} H^{(1)}_\la K^\la_{\mu 1^{n-k}}
= \frac{n!}{(n-k)!} \frac{K^\la_{\mu 1^{n-k}}}{K^\la_{1^n}} = \Ko_\mu(\la),\]
as claimed in~\cref{SubsectIntroJack}.

\subsection{Shifted symmetric functions}
\label{SubSecDefShSym}
In this section, we formally introduce the notion of $\a$-shifted symmetric functions
and present a few useful facts about them.
Our presentation mainly follows the one in \cite[Subsection 2.2]{LassalleConjecturePQ}.
\begin{definition}
    An $\alpha$-shifted symmetric function $\F$ is a sequence $(\F_N)_{N \ge 1}$
    such that 
    \begin{itemize}
        \item For each $N \ge 1$, $\F_N$ is a polynomial in $N$ variables
            $x_1, \cdots, x_N$ with coefficients in $\QQ(\a)$
            that is symmetric in $x_1-1/\alpha$, $x_2-2/\alpha$, \ldots,
            $x_N-N/\alpha$.
        \item We have the stability property: for each $N \ge 1$,
            \[\F_{N+1} (x_1,\dotsc,x_N,0) = \F_N(x_1,\dotsc,x_N).\]
        \item We have $\sup_{N \ge 1} \deg(\F_N) < \infty$.
    \end{itemize}
\end{definition}

An example of $\a$-shifted symmetric function is the following,
where we omit the index $N$ for readability:
\begin{equation}
    \psump_k^\star(x_1,\dotsc,x_N)=\sum_{i \ge 1}^N
\big[ (\a x_i-i +1/2)^k - (-i +1/2)^k \big]. 
\label{EqPkStar}
\end{equation}
It is easily shown that $(\psump_k^\star)$ is an algebraic basis
of the $\a$-shifted symmetric function ring \cite[Corollary 1.6]{OkounkovOlshanskiShiftedSchur}.
As usual for power-sums, if $\nu=(\nu_1,\dots,\nu_h)$ is a partition,
we denote $\psump_\nu^\star=\psump_{\nu_1}^\star \dots \psump_{\nu_h}^\star$.
This is a linear basis of the $\a$-shifted symmetric function ring.

Let $F$ be an $\a$-shifted symmetric function and 
$\la=(\la_1,\dotsc,\la_\ell)$ a Young diagram
with $\ell$ rows.
Then, as explained in the introduction, we define
\[F(\la) = F_\ell(\la_1,\dotsc,\la_\ell). \]
The polynomial $F_\ell$ in $\ell$ variables is determined
by its values on non-increasing lists of positive integers,
that is, Young diagrams with $\ell$ rows,
so that $F$ is determined by its values on all Young diagrams.
From this simple remark, we can identify the algebra of 
$\a$-shifted symmetric functions with a subalgebra of functions on Young diagrams.
We use this identification repeatedly in this paper without further mention.
\bigskip

Following Okounkov and Olshanski \cite{OkounkovOlshanskiShiftedJack}
--- beware that they are working with the $P$ normalization of Jack polynomials ---
we define the shifted Jack polynomials $\shJack_\mu$ as the unique $\a$-shifted
symmetric function of degree at most $|\mu|$ such that, for any partition $\la$ with $|\la| \le |\mu|$,
\begin{equation}
\shJack_\mu(\la) = \begin{cases}
    \frac{H_\la^{(\a)} {H'}_\lambda^{(\alpha)}}{\a^{|\mu|}}  & \text{if }\la=\mu, \\
    0 & \text{otherwise.}
\end{cases}
\label{eq:def_ShJack_Vanishing}
\end{equation}
The existence and uniqueness of such a function is not obvious.
Indeed, if we look for solutions of \eqref{eq:def_ShJack_Vanishing} under the form
\[\shJack_\mu=\sum_{\nu:|\nu| \le |\mu|}  a_\nu \psump_\nu^\star,\]
then \eqref{eq:def_ShJack_Vanishing} becomes a linear system of equations,
with as many indeterminates $(a_\nu)_{|\nu| \le |\mu|}$ as equations 
(there is one equation for each $\la$ with $|\la| \le |\mu|$).
We should therefore check that this system is non-degenerate.
This was done by Knop and Sahi, using an inductive argument,
in a much more general context; see \cite{KnopSahiShiftedSym}.

From this definition, the fact (advertised in the introduction) that
the top component of $\shJack_\mu$ is $J_\mu^{(\a)}$ 
is far from being obvious.
This is again a result of Knop and Sahi \cite{KnopSahiShiftedSym}.
\bigskip

The functions $\shJack_\mu$ are by definition $\a$-shifted symmetric.
From \cite[Proposition 2]{LassalleConjecturePQ},
we know that $\Cha_\mu$ is also $\a$-shifted symmetric.
Finally, the $\a$-shifted symmetry of $\Koa_\mu$
is established in the next subsection.

\subsection{Change of bases in shifted symmetric function ring}
\label{Subsect:change_bases}
\begin{proposition}
    Let $\mu$ be a partition of $k$.
    As functions on Young diagrams, one has:
    \[   \Koa_\mu = \sum_{\nu \vdash k}
        \frac{L_{\nu,\mu}}{z_\nu} \Cha_\nu. \]
    In particular, $\Koa_\mu$ is a shifted symmetric function.
    \label{PropKoaOnCha}
\end{proposition}
\begin{proof}
    Fix a Young diagram $\la$ of size $n$. 
       We want to show that 
            \[   \Koa_\mu(\la) = \sum_{\nu \vdash k}
                        \frac{L_{\nu,\mu}}{z_\nu} \Cha_\nu(\la). \]
    If $n<k$, both sides of the equality above are zero by definition,
    so let us focus on the case $n \ge k$.
    From \cref{eq:jack-characters,EqPowerOnMon}, we get
\[
J_\la^{(\a)}=\sum_{\tau \vdash n} 
\theta_{\tau}^{(\a)}(\la)\, \psump_{\tau}
= \sum_{\tau \vdash n} 
 \sum_{\rho \vdash n} 
 \theta_{\tau}^{(\a)}(\la)\, L_{\tau,\rho}\, M_\rho.
\]
Extracting the coefficient of $M_{\mu 1^{n-k}}$,
 we get
  \[\hatK^{\lambda,(\a)}_{\mu 1^{n-k}} = \sum_{\tau \vdash n}\theta_{\tau}^{(\a)}(\la)\, L_{\tau,\mu 1^{n-k}}.\]
  Using the definition of $\Koa_\mu$, given in \cref{EqDefKoA}, this yields:
  \[
\Koa_\mu(\la) = \frac{1}{(n-k)!} \hatK^{(\a)}_{\mu 1^{n-k}} (\lambda) =
\frac{1}{(n-k)!} \sum_{\tau \vdash n} \theta_{\tau}^{(\a)}(\la) \, L_{\tau,\mu 1^{n-k}}.
\]
But $L_{\tau,\mu 1^{n-k}} \neq 0$ only if $\mu 1^{n-k} $ refines $\tau$.
In particular, if this is the case, $\tau$ must be equal to $\nu 1^{n-k}$
for some partition $\nu$ of $k$.
Therefore,
\[\Koa_\mu(\la) = \frac{1}{(n-k)!} \sum_{\nu \vdash k} \theta_{\nu 1^{n-k}}^{(\a)}(\la)
L_{\nu 1^{n-k},\mu 1^{n-k}}.\]
From \cref{EqDefChA,EqLAddNmKUn},
we obtain:
\[
\Koa_\mu(\la) = \frac{(m_1(\nu)+n-k)!}{m_1(\nu)! \, (n-k)!} 
\sum_{\nu \vdash k} L_{\nu,\mu} \theta_{\nu 1^{n-k}}^{(\a)}(\la)
 = \sum_{\nu \vdash k} \frac{L_{\nu,\mu}}{z_\nu} \Cha_\nu(\la). \qedhere
\]
\end{proof}

\begin{proposition}
    Let $\mu$ be a partition of $k$.
    As functions on Young diagrams, one has:
    \[\shJack_\mu = \sum_{\nu \vdash k} 
    \theta_\nu^{(\a)}(\mu) \a^{-k+\ell(\nu)} \Cha_\nu. \]
    \label{PropShJackOnCha}
\end{proposition}
\begin{proof}
    We use the work of Lassalle \cite{LassalleConjecturePQ}.
    In this paper, Lassalle constructs a linear isomorphism that he denotes $f \mapsto f^\#$
    between usual symmetric functions and $\a$-shifted symmetric functions.
    He also proves
    \begin{align*}
        (J^{(\a)}_\mu)^\# &= \shJack_\mu \qquad \text{\cite[beginning of Section 3]{LassalleConjecturePQ}};\\
        \psump_\mu^\# &= \a^{-k+\ell(\nu)} \Cha_\nu \qquad \text{\cite[Proposition 2]{LassalleConjecturePQ}}.
    \end{align*}
    In fact, \cite[Proposition 2]{LassalleConjecturePQ}
    only gives the second equality when evaluated on
    a partition $|\la|$ of size at least $|\nu|$, but it is
    straightforward to check that both sides are equal to $0$
    when evaluated on a partition $|\la|$ of size smaller than $|\nu|$.
    Applying the linear map $f \mapsto f^\#$ to \cref{eq:jack-characters}
    gives the proposition.
\end{proof}

\begin{remark}
    The case $\a=1$ of this proposition is a reformulation of
    a result of Okounkov and Olshanski 
    \cite[Equation (15.21)]{OkounkovOlshanskiShiftedSchur}.
\end{remark}

We now observe that, for any partition $\nu$ of size $k$, one has
$L_{\nu,(k)}=1$, while $\theta_\nu^{(\a)}( (k))=\frac{k!}{z_\nu} \a^{k-\ell(\nu)}$
\cite[Section VI.10, Example 1]{Macdonald1995}.
Comparing with \cref{PropKoaOnCha,PropShJackOnCha} for $\mu=(k)$, we get
\begin{equation}
  \shJack_{(k)}=k! \cdot \Koa_{(k)}.
  \label{Eq:ShJKoa_OnePart}
\end{equation}

Finally, in the case $\a=1$, we will need a direct relation between
the $\Ko_\mu$ basis and the shifted Schur basis.
\begin{proposition}
    Let $\mu$ be a partition of $k$.
    As functions on Young diagrams, one has:
    \[\Ko_\mu = \sum_{\nu \vdash k} 
       K^\nu_\mu \ShSchur_\nu.\]
    \label{PropKoOnShSchur}
\end{proposition}
\begin{proof}
  For any diagram $\la$, we have:
\begin{equation}
\Ko_\mu(\lambda)
= (n)_k \dfrac{  K^\lambda_{{\mu1^{n-k}}} }{  K^\lambda_{1^n}   }
= \sum_{\nu \vdash k}  K^\nu_\mu \dfrac{ (n)_k K^{\lambda/\nu}_{1^n} }{ K^\lambda_{1^n}   }
= \sum_{\nu \vdash k}  K^\nu_\mu \ShSchur_\nu(\lambda)
\end{equation}
where $K^{\lambda/\nu}_{1^n}$ is the number of standard Young tableaux of skew shape $\lambda/\mu$.
Indeed, the first equality follows from the definition of $\Ko_\mu$, while the second equality uses the fact that
a semi-standard Young tableau of shape $\lambda$ and type $\mu1^{n-k}$ can be decomposed as
a semi-standard Young tableau of shape $\nu$ and type $\mu$, together with a standard Young tableau of
skew shape $\lambda/\nu$. The last equality comes from \cite[Eq. (0.14)]{OkounkovOlshanskiShiftedSchur},
which states that
\[
 \frac{K^{\lambda/\nu}_{1^n}}{K^\lambda_{1^n}} = \frac{\ShSchur_\nu(\lambda) }{(n)_k}.\qedhere
\]
\end{proof}

\subsection{Shifted symmetric functions are polynomials in multirectangular coordinates}

The purpose of this Section is to establish the statement in its title.
The proof is a straight-forward adaptation of some arguments
of Ivanov and Olshanski for $\a=1$ \cite[Sections 1 and 2]{IvanovOlshanski2002}.
\bigskip

For any positive integer $N$,
we consider the following expression: 
\[ \phi_N(x_1,\dotsc,x_N;z) = \prod_{i=1}^N \frac{z +i-1/2}{z-\a x_i +i-1/2 }. \]
The logarithm of $\phi_N$ has the following expansion around $z=\infty$:
\[ \ln \big(\phi_N(x_1,\dotsc,x_N;z)\big) =
\sum_{k \ge 1} \psump_k^\star(x_1,\dotsc,x_N) z^{-k}/k,\]
where $p_k^\star$ is defined in \cref{EqPkStar}.
Recall that $(\psump_k^\star)_{k \ge 1}$
form an algebraic basis of the $\a$-shifted symmetric function ring.
We also consider the following quotient:
\[\Psi_N(x_1,\dotsc,x_N;z) = \frac{\phi_N(x_1,\dotsc,x_N;z-1/2)}{\phi_N(x_1,\dotsc,x_N;z+1/2)}\]
Again, its logarithm can be expanded around $z=\infty$:
there exist $\a$-shifted symmetric functions $(\pst{k})_{k \ge 1}$ such that
\[\ln \big(\Psi_N(x_1,\dotsc,x_N;z)\big) =
\sum_{k \ge 1} \pst{k}(x_1,\dotsc,x_N) z^{-k}/k.\]
The relation between $(\psump_k^\star)_{k \ge 1}$ and $(\pst{k})_{k \ge 1}$
is discussed in \cite[Proposition 2.7 and Corollary 2.8]{IvanovOlshanski2002}.
In particular, it is easily seen that $\pst{1}=0$
and that $(\pst{k})_{k \ge 2}$ is an algebraic basis 
of the $\a$-shifted symmetric function ring.
\medskip

If $\lambda$ is a Young diagram
with $\ell$ rows, we define $\Psi(\la;z)=\Psi_\ell(\la_1,\dotsc,\la_\ell;z)$.
The following statement is an $\a$-analog of \cite[Proposition 2.6]{IvanovOlshanski2002}.
\begin{proposition}
    Let $\pvec=(p_1,\dotsc,p_d)$ and $\rvec=(r_1,\dotsc,r_d)$ be two lists of non-negative integers.
    Recall that $\rvec^\pvec$ denotes the Young diagram in multirectangular
    coordinates $\rvec$ and $\pvec$; see \cref{FigDiagMulti}.
    Then
    \[\Psi(\rvec^\pvec;z)= \frac{z\, \prod_{s=1}^d \big( z-\a(r_s+\dotsb+r_d) + p_1 +\dotsb +p_s \big) }
    {\prod_{s=1}^{d+1} \big( z-\a(r_{s}+\dotsb+r_d)+p_1 +\dotsb +p_{s-1} \big) }.\]
    \label{PropPsiCorners}
\end{proposition}
\begin{proof}
    Straight from the definitions, we have, that, for any Young diagram $\la$ with $\ell$ rows,
    \begin{multline}
        \Psi(\la;z) = \frac{\phi_\ell(\la_1,\dotsc,\la_\ell;z-1/2)}
    {\phi_\ell(\la_1,\dotsc,\la_\ell;z+1/2)}
    =\frac{\prod_{i=1}^\ell (z +i-1)/(z-\a \la_i +i-1)}
    {\prod_{i=1}^\ell (z +i)/(z-\a \la_i +i)}\\
    =\frac{z}{z+\ell} \, \prod_{i=1}^\ell
    \frac{z-\a \la_i +i}{z-\a \la_i +i-1}.
    \label{EqPsiLa}
\end{multline}
Set now $\la=\rvec^\pvec$.
By definition of multirectangular coordinates,
for $i$ between $1$ and $p_1$, one has $\la_i=r_1+\dotsb+r_d$.
The product of the factors in \cref{EqPsiLa} indexed by these values simplifies as follows:
\[
\prod_{i=1}^{p_1} \frac{z-\a \la_i +i}{z-\a \la_i +i-1} = 
\frac{z- \a(r_1+\dotsb+r_d)+p_1}{z-\a(r_1+\dotsb+r_d)}.
\]
More generally, let $s$ be an integer with $1 \le s \le r$.
If $p_1+\cdots+p_{s-1} < i \le p_1 +\dotsb+p_s$, then $\la_i=r_s+\dotsb+r_d$.
Here is the corresponding partial product in \cref{EqPsiLa}:
\[
\prod_{i=p_1+\cdots+p_{s-1}+1}^{p_1+\cdots+p_s} \frac{z-\a \la_i +i}{z-\a \la_i +i-1} =
\frac{z- \a(r_s+\cdots+r_d)+p_1+\cdots+p_s}{z-\a(r_s+\cdots+r_d) +p_1+\cdots+p_{s-1}}.
\]
Putting everything together we get that:
\[
\Psi(\rvec^\pvec;z) =\frac{z}{z+\ell} \, \prod_{s=1}^d
\frac{z- \a(r_s+\cdots+r_d)+p_1+\cdots+p_s}{z-\a(r_s+\cdots+r_d) +p_1+\cdots+p_{s-1}}.
\]
Since the number $\ell$ of rows of the Young diagram $\la=\rvec^\pvec$
is equal to $p_1+\cdots+p_d$, this ends the proof.
\end{proof}
\begin{corollary}
    Let $F$ be an $\alpha$-shifted symmetric function and $d$ a positive integer.
    Then $F(\rvec^\pvec)$ is a polynomial in $p_1,\dotsc,p_d,r_1,\dotsc,r_d$
    with coefficients in $\QQ(\a)$.
    \label{CorolShiftSymPoly}
\end{corollary}
\begin{proof}
    Since $(\pst{k})_{k \ge 2}$ is an algebraic basis of the $\a$-shifted symmetric function ring,
    it is enough to prove this corollary for $F=\pst{k}$. But, by definition,
    \[\ln \big(\Psi(\rvec^\pvec;z)\big) = \sum_{k \ge 1} \pst{k}(\rvec^\pvec) z^{-k}/k.\]
    Therefore, it follows from \cref{PropPsiCorners} that $\pst{k}(\rvec^\pvec)$
    is a polynomial in the variables $p_1,\dotsc,p_d,r_1,\dotsc,r_d$
    with coefficients in $\QQ(\a)$.
\end{proof}

\section{The Schur case: proof of  \texorpdfstring{\cref{thm:mainshifted}}{the case a=1}}
\label{SectAlpha1}

\subsection{Set-partitions}
\label{SubsecSetPartitions}
We first recall a few basic facts about set-partitions.

A \emph{set-partition} $S$ of a ground set $E$ is
a set $\{S_1,\dots,S_s\}$ of disjoint subsets of $E$
whose union is $E$.
The sets $S_i$ are called \emph{blocks} (for $1 \le i \le s$)
and $s=|S|$ is the number of blocks.

Set-partitions of a given ground set $E$
are endowed with a natural partial order,
called the \emph{refinement order}.
By definition, $\tilde{S}\leq S$ if each block of $\tilde{S}$
is included in some block of $S$.
We then say that $\tilde{S}$ is \emph{finer} than $S$
or $S$ is \emph{coarser} than $\tilde{S}$.
It is well-known that the set of set-partitions of $E$
is a lattice with respect to this order.

We will use the following easy facts.
\begin{itemize}
    \item Assume $\tilde{S}\leq S$ and let $S_i$ be a block of $S$.
        Then the set of blocks of $\tilde{S}$ that are included in $S_i$
        is a set-partition of $S_i$.
        We call this set-partition \emph{induced} by $\tilde{S}$ on $S_i$
        and denote it by $\tilde{S}/ S_i$.
        In particular $|\tilde{S}/S_i|$ is the number of blocks of $\tilde{S}$ included in $S_i$.
    \item Conversely, if for each block $S_i$ of $S$,
        we choose a set-partition $\tilde{S}^i$ of $S_i$,
        then the union of the $\tilde{S}^i$ is a set-partition $\tilde{S}$ of $E$
        that is finer than $S$.
    \item Fix a set-partition $S$ of $E$.
      There is a canonical bijection between set-partitions \hbox{$T=\{T_1,\dots,T_t\}$} of $S$
        and set-partitions of $E$ that are coarser than $S$.
        The bijection is constructed by replacing
        each block $T_j=\{S_{i_1},\dots,S_{i_j}\}$ of $T$ by the union
        $S_{i_1} \cup \dots \cup S_{i_j}$.
\end{itemize}

A function $f:S \to D$ associates by definition an element of $D$
with each block of $S$.
Note that $f$ induces canonically a map $\bar{f}$ from the ground set $E$
to $D$ which is constant on each block of $S$.
Conversely, given a function $g:E \to D$,
we denote $\Pi(g)$ the set-partition $\{g^{-1}(\{x\}), x \in g(E)\}$ of $E$.
In other words, $e$ and $e'$ are in the same part of $\Pi(g)$ if and only if
$g(e)=g'(e)$. 
If $S$ is a set-partition of $E$ and $f$ a function from $S$ to $D$,
then we have the relation: $S \le \Pi\big(\bar{f}\big)$.

\subsection{A family of polynomials indexed by pairs of permutations}
\label{SubsectN}

Let $d$ and $k$ be fixed positive integers
and $S$ and $T$ be set-partitions of $[k]$.
Two maps
\[
 v:S \to [d] \qquad w:T \to [d]
\]
are said to be \emph{compatible}
if $v(S_i) \leq w(T_j)$
whenever $S_i$ and $T_j$
are blocks of $S$ and $T$ with a non-empty intersection.

Now fix two permutations $\sigma$ and $\tau$ in $\perms_k$.
Their sets of cycles $C(\sigma)$ and $C(\tau)$
can be interpreted as set-partitions of $[k]$
by treating each cycle as a block.

Now we define $N_{\sigma,\tau}(\rvec,\pvec)$ as
\begin{equation}\label{EqFerayNDef}
 N_{\sigma,\tau}(\rvec,\pvec) = \sum_{\substack{v:C(\sigma) \to [d] \\ w :C(\tau) \to [d] \\ v, w\text{ compatible}}} \left(\prod_{\sigma_i \in C(\sigma)} p_{v(\sigma_i)} \right) \left(\prod_{\tau_j \in C(\tau)} r_{w(\tau_j)}\right).
\end{equation}
Note that $N_{\sigma,\tau}(\rvec,\pvec) = N_{\sigma',\tau}(\rvec,\pvec)$
whenever the cycles of $\sigma$ and $\sigma'$ induce the same set-partition of $[k]$.
The same statement holds for $\tau$ and $\tau'$.
Also if $\rho$ is a permutation of $k$, then
$N_{\rho^{-1} \sigma \rho ,\rho^{-1} \tau \rho}(\rvec,\pvec) = N_{\sigma,\tau}(\rvec,\pvec)$.


The relevance of this family of polynomials will be evident in the next subsection.
We will first give a second expression for $N_{\sigma,\tau}(\rvec,\pvec)$.

\begin{lemma}\label{LemFerayN2}
For $\sigma, \tau \in \perms_k$, the polynomial $N_{\sigma,\tau}(\rvec,\pvec)$ is given by
\begin{equation}\label{eq:nst}
N_{\sigma,\tau}(\rvec,\pvec) =
\sum_{\substack{ S, T \text{ set-partitions of } [k] \\ C(\sigma) \leq S \\ C(\tau) \leq T }}
\quad
\sum_{
\substack{v:S \hookrightarrow [d] \\ w : T \hookrightarrow [d] \\ v, w\text{ compatible}\\ \text{and injective}}
}
\quad
\prod_{\substack{S_i \in S \\ T_j \in T }} p^{|C(\sigma)/S_i|}_{v(S_i)} q^{|C(\tau)/T_j|}_{w(T_j)}.
\end{equation}
The middle sum is performed over all compatible \emph{injective} maps $v$, $w$
defined on $S$ and $T$, respectively.
\end{lemma} 
\begin{proof}
Consider a map $v: C(\sigma) \to [d]$.
Recall that it corresponds to a map $\bar{v}: [k] \to [d]$,
constant on blocks of $C(\sigma)$.
This map $\bar{v}$ corresponds to a unique pair $(S,v')$,
where $S$ is a set-partition of $[k]$ and $v'$ an \emph{injective}
map from $S$ to $d$.
Indeed, necessarily $S=\Pi(\bar{v})$. 

Therefore, the sum over $v:C(\sigma) \to [d]$ in \eqref{EqFerayNDef}
can be replaced by a double sum over $S$ and $v'$ as above.
Similarly, the sum over $w : C(\tau) \rightarrow [d]$ can be replaced
by a double sum over $T$ and $w'$ with $T \ge C(\tau)$ and
$w': T \rightarrow [d]$ injective.
Lastly, $v$ and $w$ are compatible if and only if $v'$ and $w'$ are,
which completes the proof.
\end{proof}

\begin{remark}
    \label{RkMax}
    In its original paper \cite{Stanley2003}, Stanely uses a slightly different
    version of multirectangular coordinates, namely $\pvec$ and $\qvec$
    where $q_i=r_i+r_{i+1} +\ldots$.
    In terms of these coordinates, the definition of $N$ is rewritten as:
\begin{equation}
 N_{\sigma,\tau}(\qvec,\pvec) =
 \sum_{\substack{v:C(\sigma) \to [d] \\ w :C(\tau) \to [d] \\ v, w\text{ max-compatible}}} 
 \left(\prod_{\sigma_i \in C(\sigma)} p_{v(\sigma_i)} \right) \left(\prod_{\tau_j \in C(\tau)} q_{w(\tau_j)}\right),
\end{equation}
where $v, w\text{ max-compatible}$ means that, for $\tau_j \in C(\tau)$,
\[w(\tau_j) = \max_{\substack{\sigma_i \in C(\sigma) \\ \sigma_i \cap \tau_j \neq \emptyset}} v(\sigma_i).\]
In particular it corresponds to the quantity $N^{\la}(\si,\tau)$ considered, \emph{e.g.}, in \cite{NousBoundsOnCharacters}.
\end{remark}

\subsection{Combinatorial formulas}
\label{SectCombForulasA1}

Given a partition $\mu \vdash k$, 
let $\pi_\mu$ denote an arbitrarily chosen permutation 
in $\perms_k$ with cycle type $\mu$.
We also arbitrarily choose a set-partition $U_\mu$,
so that the lengths of its part, sorted in decreasing order, are given by $\mu$.
Denote the sign of a permutation $\sigma$ with $\eps(\sigma)$.

\begin{proposition}\label{combFormulasInFerayN}
The polynomials $\Ch_\mu(\rvec^\pvec)$, $\ShSchur_\mu(\rvec^\pvec)$ and $\Ko_\mu(\rvec^\pvec)$
are expressed in terms of $N_{\sigma,\tau}$ as follows:
\begin{align}
\Ch_\mu(\rvec^\pvec) &= \sum_{\substack{\sigma, \tau \in \perms_k \\ \sigma \tau = \pi_\mu}} \eps(\tau)N_{\sigma,\tau}(\rvec,\pvec),
\label{EqChN}\\
    \ShSchur_\mu(\rvec^\pvec) &= \frac{1}{k!}  \sum_{ \substack{\sigma,\tau \in \perms_k} } \chi^\mu( \sigma \tau )  \eps(\tau) N_{\sigma,\tau}(\rvec,\pvec),
\label{EqShSchurN}\\
\Ko_\mu(\rvec^\pvec) &= \frac{1}{\prod_{i=1}^{\ell(\mu)} \mu_i!}
\sum_{\substack{\sigma, \tau \in \perms_k \\ C(\sigma \tau) \leq U_\mu}} \eps(\tau)N_{\sigma,\tau}(\rvec,\pvec).
     \label{EqKoN}
\end{align}
\end{proposition}
\begin{remark}
    \cref{EqKoN} is not used later in the proof.
    We give it only for completeness.
\end{remark}
\begin{proof}
\cref{EqChN} is the main result of \cite{FerayPreuveStanley};
see \cite{NousBoundsOnCharacters} for a more elementary alternate proof.

Specializing \cref{PropShJackOnCha} for $\a=1$, we get
\begin{equation}
\ShSchur_\mu(\rvec^\pvec) = \sum_{\nu \vdash k } \frac{ \chi^\mu_\nu }{z_\nu} \Ch_\nu(\lambda)  =
\sum_{\nu \vdash k } \frac{ \chi^\mu_\nu }{z_\nu} \sum_{ \substack{\sigma,\tau \in \perms_k \\ \sigma\tau = \pi_\nu}} \eps(\tau) N_{\sigma,\tau}(\rvec,\pvec).
\label{EqTechS1}
\end{equation}
On the other hand,
\begin{equation}
    \frac{1}{k!}  \sum_{ \substack{\sigma,\tau \in \perms_k} } \chi^\mu( \sigma \tau )  \eps(\tau) N_{\sigma,\tau}(\rvec,\pvec)
= \frac{1}{k!} \sum_{\pi \in \Sn_k} \chi^\mu(\pi) \left[ \sum_{ \substack{\sigma,\tau \in \perms_k \\ \sigma \tau = \pi} } \eps(\tau) N_{\sigma,\tau}(\rvec,\pvec) \right]. 
\label{EqTechSS}
\end{equation}
Since $N$ is invariant by simultaneous conjugacy of $(\sigma,\tau)$,
the expression in the bracket depends only on the conjugacy class $\pi$.
So does $\chi^\mu(\pi)$.
For each partition $\nu \vdash k$, there are $k!/z_\nu$ permutations conjugated to $\pi_\nu$.
Therefore, \cref{EqTechSS} rewrites
\begin{equation}
    \frac{1}{k!}  \sum_{ \substack{\sigma,\tau \in \perms_k} } \chi^\mu( \sigma \tau )  \eps(\tau) N_{\sigma,\tau}(\rvec,\pvec)
    = \sum_{\nu \vdash k}\frac{\chi^\mu(\pi_\nu)}{z_\nu}
    \left[ \sum_{ \substack{\sigma,\tau \in \perms_k \\ \sigma \tau = \pi_\nu} }
    \eps(\tau) N_{\sigma,\tau}(\rvec,\pvec) \right]. 
    \label{EqTechS2}
\end{equation}
Comparing \cref{EqTechS1,EqTechS2} gives \cref{EqShSchurN} --- note that $\chi^\mu_\nu=\chi^\mu(\pi_\nu)$ by definition.

Specializing \cref{PropKoaOnCha} for $\a=1$, we get
\begin{equation}
    \Ko_\mu(\rvec^\pvec) = \sum_{\nu \vdash k } \frac{ L_{\nu,\mu} }{z_\nu} \Ch_\nu(\lambda)  =
    \sum_{\nu \vdash k } \frac{ L_{\nu,\mu}}{z_\nu} 
    \sum_{ \substack{\sigma,\tau \in \perms_k \\ \sigma\tau = \pi_\nu}} \eps(\tau) N_{\sigma,\tau}(\rvec,\pvec).
\end{equation}
Using the same trick as above and denoting $\type(\sigma\tau)$ the cycle type of the permutation $\sigma\tau$,
we get
\[\Ko_\mu(\rvec^\pvec) = \frac{1}{k!} \sum_{\sigma,\tau \in S_d} \eps(\tau) 
\, L_{\type(\sigma\tau),\mu} \, N_{\sigma,\tau}(\rvec,\pvec). \]
From \cref{LemLPerm},
$L_{\type(\sigma\tau),\mu}$ is the number of functions 
$f:[k] \to [\ell(\mu)]$ with $C(\sigma\tau) \le \Pi(f)$
and such that each number $i$ has exactly $\mu_i$ pre-images by $f$.
Hence
\begin{equation}
    \Ko_\mu(\rvec^\pvec) = \frac{1}{k!} 
\sum_{ \substack{f:[k] \to [\ell(\mu)] \\ |f^{-1}(i)|=\mu_i}}
\left[ \sum_{ \substack{ \sigma,\tau \in S_d \\ C(\sigma\tau) \le \Pi(f) }}
 \eps(\tau) \, N_{\sigma,\tau}(\rvec,\pvec) \right].
 \label{TechKo}
 \end{equation}
 Fix some function $f_0$ in the first sum.
 Any other function $f$ with $|f^{-1}(i)|=\mu_i$ (for all $i \le \ell(\mu)$)
 writes as $f_0 \circ \rho$ for some permutation $\rho \in S_k$.
 Then a pair $(\sigma,\tau)$ fulfills $C(\sigma\tau) \le \Pi(f)$
 if and only if $(\sigma',\tau') \coloneqq (\rho \sigma \rho^{-1}, \rho\tau \rho^{-1})$ fulfills
 $C(\sigma'\tau') \le \Pi(f_0)$.
 As $\eps(\tau)$ and $N_{\sigma,\tau}(\rvec,\pvec)$ are invariant by simultaneous conjugation
 of $\sigma$ and $\tau$, the expression in the bracket in \eqref{TechKo}
 does not depend of $f$.
 The number of functions $f$ with $|f^{-1}(i)|=\mu_i$ (for all $i \le \ell(\mu)$)
 is $k!/\prod_{i=1}^{\ell(\mu)} \mu_i!$, so that
 \eqref{TechKo} can be rewritten as:
 \[\Ko_\mu(\rvec^\pvec) = \frac{1}{\prod_{i=1}^{\ell(\mu)} \mu_i!} 
\left[ \sum_{ \substack{ \sigma,\tau \in S_d \\ C(\sigma\tau) \le \Pi(f_0) }}
 \eps(\tau) \, N_{\sigma,\tau}(\rvec,\pvec) \right].\]
Since the only condition on $f_0$ is that the block sizes of $\Pi(f_0)$
are given by $\mu$,
we can assume without loss of generality that $\Pi(f_0)=U_\mu$.
This conclude the proof of \cref{EqKoN}.
\end{proof}

\bigskip

\subsection{Positivity of  \texorpdfstring{$\ShSchur_\mu$}{shifted Schur} and  \texorpdfstring{$\Ko_\mu$}{Kostka} in the falling factorial basis}
\label{SubsectProofAUn}
Let us start by $\ShSchur_\mu$.
We first express the shifted Schur polynomial $\ShSchur_\mu$ as a sum of smaller pieces,
each of which will be proved to be nonnegative in the falling factorial basis.
For this, we introduce a new family of polynomials. Let $\mu \vdash k$ and $S$, $T$ be set-partitions of $[k]$
with respectively $s$ and $t$ parts. Define $A^\mu_{S,T}$ as
\begin{equation}
A^\mu_{S,T}(x_1,\dotsc,x_s,y_1,\dotsc,y_t) =  \sum_{ \substack{\sigma,\tau \in \perms_k \\ C(\sigma) \leq S \\ C(\tau) \leq T }}
\chi^\mu( \sigma \tau )  \epsilon(\tau)
\prod_{\substack{S_i \in S \\ T_j \in T }} x^{|C(\sigma)/S_i|}_{i} y^{|C(\tau)/T_j|}_{j}.
\label{EqDefA}
\end{equation}
Note that combining \cref{EqShSchurN} and \cref{eq:nst} 
and rearranging sums yields:
\begin{equation}
\ShSchur_\mu (\rvec^\pvec) = \frac{1}{k!} \sum_{\substack{ S, T \\ \text{set-partitions}\\ \text{of } [k]}}
\sum_{\substack{v:S \hookrightarrow [d] \\ w : T \hookrightarrow [d] \\ v, w\text{ compatible}\\ \text{and injective}}}
A^{\mu}_{S,T}\left( p_{v(S_1)},\dotsc,p_{v(S_s)}, r_{w(T_1)},\dotsc,r_{w(T_t)} \right).
\label{eq:ShSchur_A}
\end{equation}
Observe that if $A^\mu_{S,T}(x_1,\dotsc,x_s,y_1,\dotsc,y_t)$ is non-negative
in the falling factorial basis, then so is $\ShSchur_\mu$ --- 
it is crucial here that the inner sum is over injective functions,
so that we always evaluate $A_{S,T}^\mu$ in distinct variables $\bm{p}$ and 
$\bm{r}$;
otherwise, the expression of  $A_{S,T}^\mu$ in the falling factorial basis
in $\bm{x}$ and $\bm{y}$
would not directly yield the expression of $\ShSchur_\mu$
in the falling factorial basis in $\bm{p}$ and $\bm{r}$.

Hence, it suffices to examine $A^\mu_{S,T}$ in the falling factorial basis.
\begin{lemma}
    Fix some set-partitions $S$ and $T$ of $[k]$ and an integer partition $\mu \vdash k$.
    Then
\begin{equation}
A^\mu_{S,T}(x_1,\dotsc,x_s,y_1,\dotsc,y_t) = \sum_{ \substack{ \tilde{S} \leq S \\ \tilde{T} \leq T }} \left(
\sum_{ \substack{\sigma,\tau \in \perms_k \\ C(\sigma) \leq \tilde{S} \\ C(\tau) \leq \tilde{T} }}
\chi^\mu( \sigma \tau )  \epsilon(\tau) \right)\,
\left( \prod_{\substack{S_i \in S \\ T_j \in T }} (x_i)_{|\tilde{S}/S_i|} (y_i)_{|\tilde{T}/T_i|} \right).
\label{EqAFF}
\end{equation}
\label{LemAFF}
\end{lemma}
\begin{proof}
    Recall that the \emph{Stirling number of the second kind} $\stirling{j}{r}$
    counts set-partitions of a ground set $E$ of cardinality $j$ into $r$ blocks.
    Besides, it is well-known that, for any $j \ge 0$
    \[ x^j = \sum_{r=1} \stirling{j}{r} (x)_r = \sum_S (x)_{|S|},\]
    where the second sum runs over set-partitions of $E$.

    Now consider a permutation $\sigma$ with $C(\sigma) \le S$
    and fix $i$ between $1$ and $S$.
    Let $j=|C(\sigma)/S_i|$; then $\sigma$
    has exactly $j$ cycles $C_{i,1},\dots,C_{i,j}$
         that are included in $S_i$ (we see $C_{i,1},\dots,C_{i,j}$ as sets).
             Let $\tilde{S}^i$ be a set-partition of $\{C_{i,1},\dots,C_{i,j}\}$.
    Equivalently, $\tilde{S}^i$ can be seen as a set-partition of $S_i$
    that is coarser than $\{C_{i,1},\dots,C_{i,j}\}$.
    We have
    \[x_i^{|C(\sigma)/S_i|} = \sum_{\tilde{S}^i} (x_i)_{|\tilde{S}^i|},\]
    where the sum runs over set-partitions $\tilde{S}^i$ as above.
    Multiplying over $i$, we get:
    \[ \prod_{i=1}^s x_i^{|C(\sigma)/S_i|} = \sum_{(\tilde{S}^1,\dots,\tilde{S}^s)}\
    \prod_{i=1}^s (x_i)_{|\tilde{S}^i|}.\]
    Here the sum runs over $s$-tuples $(\tilde{S}^1,\dots,\tilde{S}^s)$,
    where $\tilde{S}^i$ is a set-partition of $S_i$ coarser than $C(\sigma)/S_i$.
    Such an $s$-tuple can be interpreted as a set-partition $\tilde{S}$ of $[k]=\bigsqcup S_i$,
    that is finer than $S$ but coarser than $C(\sigma)$
    --- see \cref{SubsecSetPartitions}.
    Finally,
    \[ \prod_{i=1}^s x_i^{|C(\sigma)/S_i|} =
    \sum_{\substack{ \tilde{S} \text{ set-partition of }[k] \\ C(\sigma) \le \tilde{S} \le S }}\,
    \prod_{i=1}^s (x_i)_{|\tilde{S}/S_i|}. \]
    Similarly,
    \[ \prod_{j=1}^t x_j^{|C(\tau)/T_j|} =
    \sum_{\substack{ \tilde{T} \text{ set-partition of }[k] 
    \\ C(\tau) \le \tilde{T} \le T}}\,
    \prod_{j=1}^t (x_j)_{|\tilde{T}/T_j|}. \]
    Plugging these two formulas in \cref{EqDefA} and changing the order of summation,
    we get \cref{EqAFF}.
\end{proof}

Note that the inner sum in \cref{EqAFF} does not depend on $S$ and $T$, but only on $\tilde{S}$ and $\tilde{T}$.
We prove that it is nonnegative in the following lemma.
If $S$ is a set-partition of $[k]$, let $\Sn_S \subseteq \Sn_k$ denote the subgroup
$\Sn_S = \{ \sigma \in \Sn_k : C(\sigma) \leq S \}$. 
\begin{lemma}\label{lem:nonnegsum}
For every $\mu \vdash k$  and pair of set-partitions $S$ and $T$ of $[k]$, the integer $B^\mu_{S,T}$ defined as
\begin{equation}
B^\mu_{S,T}=\sum_{ \substack{\sigma \in \Sn_S \\ \tau \in \Sn_T }}
\chi^\mu( \sigma \tau )  \epsilon(\tau)
\label{EqDefB}
\end{equation}
is non-negative.
\end{lemma}
\begin{proof}
Consider the symmetric group algebra $\setC[\Sn_k]$ and let
\[
X_\mu = \sum_{\pi \in \Sn_k} \chi^\mu(\pi^{-1}) \pi, \quad
Y_{S} = \sum_{\substack{\sigma \in \Sn_S}} \sigma \quad
 \text{ and } \quad
\Sigma_{T} = \sum_{\substack{\tau \in \Sn_T}} \epsilon(\tau)\tau.
\]
Then $B^\mu_{S,T} = [\id]X_\mu Y_S \Sigma_T$, 
where $[\id] A$ denotes the coefficient of the identity permutation in $A \in \setC[\Sn_k]$.
Note that all three elements considered above are quasi-idempotents:
indeed $Y_S^2 = |\Sn_S|\, Y_S$
and $\Sigma_T^2 = |\Sn_T|\, \Sigma_T$, 
while a classical result of representation theory --- see \emph{e.g.}, \cite[Exercise 6.4]{Serre} ---
asserts that $X_\mu^2 = \big(n! / \chi^\mu_{1^k}\big) \, X_\mu$.
\medskip

Furthermore, $X_\mu$ is a central element of $\setC[\Sn_k]$
and hence commutes with $Y_S$ and $\Sigma_T$.
Also using the identity $[\id]AB = [\id]BA$ for elements $A$ and $B$ in $\setC[\Sn_k]$, we can write:
\[
B^\mu_{S,T} = [\id]X_\mu Y_S \Sigma_T 
= \frac{\chi^\mu_{1^k}}{n!\, |\Sn_S|\,|\Sn_T|} [\id]X^2_\mu Y^2_S \Sigma^2_T 
	= \frac{\chi^\mu_{1^k}}{n!\, |\Sn_S|\,|\Sn_T|} [\id] (\Sigma_T X_\mu Y_S)  (Y_S X_\mu \Sigma_T).
    \]
We set $A^*=\sum_{g \in \Sn_k} \overline{c_g}\, g^{-1}$ if $A = \sum_{g \in \Sn_k} c_g \, g$.
Note that this operation is an anti-morphism, \emph{i.e.}, $(AB)^*=B^* A^*$.
Besides, for any element $A$ in $\setC[\Sn_k]$, one has: $[\id]AA^* = \sum_{g \in A} |c_g|^2 \geq 0$.
It is now evident that $Y_S = Y_S^*$, $\Sigma_T = \Sigma_T^*$ and $X_\mu = X_\mu^*$.
Finally
\[
B^\mu_{S,T} = \frac{\chi^\mu_{1^k}}{n!\, |\Sn_S|\,|\Sn_T|} [\id] (\Sigma_T X_\mu Y_S)  (Y^*_S X^*_\mu \Sigma^*_T) 
	= \frac{\chi^\mu_{1^k}}{n!\, |\Sn_S|\,|\Sn_T|} [\id] (\Sigma_T X_\mu Y_S)  (\Sigma_T X_\mu Y_S)^*,
    \]
which clearly is non-negative.
\end{proof}
\begin{proof}
    [Proof of \cref{thm:mainshifted}]
    \cref{LemAFF,lem:nonnegsum} imply that the polynomials
    $A_{S,T}^\mu$ have nonnegative coefficients in the falling factorial basis.
    But, from \cref{eq:ShSchur_A} (see also discussion right after this equation),
    $\ShSchur_\mu(\rvec^\pvec)$ is a nonnegative
    linear combination of such functions.
    Thus, it also has nonnegative coefficients in the falling factorial basis.

From \cref{PropKoOnShSchur}, $\Ko_\mu$ is a nonnegative linear combination
of $(\ShSchur_\nu)_{\nu \vdash |\mu|}$.
Therefore $\Ko_\mu$ also has nonnegative coefficients in the falling factorial basis.
\end{proof}

\subsection{Discussion}

As explained in the proof of \cref{thm:mainshifted} above,
the nonnegativity of $\Ko_\mu$ follows from the nonnegativity of $\ShSchur_\mu$.
Nevertheless, it is natural to try to give a direct proof,
starting from \cref{EqKoN}, along the same lines as for $\ShSchur_\mu$.
This raises the following question.

\begin{question}\label{question:bad}
For any triple $(S,T,U)$ of set-partitions of $[k]$, is it true that
\begin{equation}
    \sum_{\substack{\sigma \in \Sn_S \\ \tau \in \Sn_T \\ \sigma \tau \in \Sn_U}} \epsilon(\tau) \ge 0 ?
\end{equation}
\end{question}
Unlike \cref{lem:nonnegsum}, this does not hold:
computer exploration has given us a minimal counter-example for $k=8$ ---
see \cite{MathOverflow194852} for details.


\section{Formulas for  \texorpdfstring{$Z^\sh_\mu$}{shifted zonal} and  \texorpdfstring{$\Ko^{(2)}_\mu$}{Kostka} in the falling factorial basis}
\label{SectAlpha2}

In this section,
we give formulas for $Z^\sh_\mu$ and $\Ko^{(2)}_\mu$ in the falling factorial basis.
Although these formulas are very much similar to the case $\a=1$,
we have not been able to use them to prove the case $\a=2$ of \cref{ConjMain}.

As often when considering Jack polynomials for $\a=2$,
permutations should be replaced by pair-partitions
and characters of the symmetric groups 
 by zonal spherical functions of the pair $(S_{2k},H_k)$ ---
  $H_k$ denotes here the hyperoctahedral group, seen as a subgroup of $S_{2k}$,
  the definition is given below.


\subsection{Pair-partitions}

\begin{definition}
A pair-partition $P$ of $[2k]=\{1,\ldots,2k\}$ is a set of pairwise disjoint
two-element sets, such that their (disjoint) union is equal to
$[2k]$. 
The set of pair-partitions of $[2k]$ is denoted $\PP{k}$.
\end{definition}
The simplest example is the \emph{first} pair-partition, which plays a
par\-ti\-cu\-lar role:
\begin{equation}
    \label{eq:first-pair-partition}
    S_\star=\big\{ \{1,2\},\{3,4\},\dots,\{2k-1,2k\}\big\}.
\end{equation}

Let us consider two pair-partitions $S_1,S_2$ of the same set $[2k]$.
We denote by $\U(S_1,S_2)$ the join of $S_1$ and $S_2$ in the set-partition lattice,
\emph{i.e.}, the finest set-partition that is coarser than $S_1$ and $S_2$.
Each block of $\U(S_1,S_2)$ is a disjoint union of blocks of $S_1$,
and hence has even size.
Let $2\ell_1\geq 2\ell_2\geq\cdots$ be the ordered sizes of these blocks.
The (integer) partition $(\ell_1,\ell_2,\dots)$ of $k$
is called the \emph{type} of the pair $(S_1,S_2)$.

\begin{example}
    Consider $S_1 = \big\{ \{1,2\},\{3,4\},\{5,6\} \big\}$
    and $S_2 = \big\{ \{1,3\},\{2,4\},\{5,6\} \big\}$.
    Then $\U(S_1,S_2)= \big\{ \{1,2,3,4\},\{5,6\} \big\}$
    and $(S_1,S_2)$ has type $(2,1)$.
\end{example}

Permutations $\sigma$ in $\Sn_{2k}$ act on pair-partitions $S$ of $\PP{k}$ as follows:
\[ \sigma \cdot \big\{ \{i_1,j_1\}, \dots, \{i_k,j_k\} \big\} =
 \big\{ \{\sigma(i_1),\sigma(j_1)\}, \dots, \{\sigma(i_k),\sigma(j_k)\} \big\}.  \]
 
 Note that the type of a pair of pair-partitions is invariant by simultaneous action of the symmetric group,
 \emph{i.e.}, $(S_1,S_2)$ and $(\sigma \cdot S_1,\sigma \cdot S_2)$ have the same type.

\subsection{The Gelfand pair  \texorpdfstring{$(\Sn_{2k},H_k)$}{} and its zonal spherical functions}~\\
\label{ssec:zonal_bakground}
Representation-theoretical arguments from \cref{SectAlpha1} should be replaced
for $\a=2$ by properties of the so-called \emph{Gelfand pair} $(\Sn_{2k},H_k)$.
Here, we review the necessary material about the theory of finite Gelfand pairs
and the particular case $(\Sn_{2k},H_k)$.
All statements of this Section can be found with proofs in \cite[Sections VII,1 and VII,2]{Macdonald1995}.

Let us consider the symmetric group $\Sn_{2k}$ on $2k$ elements.
Recall that permutations in $\Sn_{2k}$ act on the set $\PP{k}$ of pair-partitions.
The set of permutations that fix $S_\star$ is a subgroup of $\Sn_{2k}$
called hyperoctahedral group and denoted by $H_k$.
It is well-known that $H_k$ is isomorphic to the group of signed permutations.
In particular, $|H_k|=2^k k!$.

Here we will be interested in the double-class algebra $\setC[H_k \backslash \Sn_{2k} / H_k]$,
that is the subalgebra of $\setC[\Sn_{2k}]$ that is invariant by left or right multiplication by
an element of $H_k$.
This algebra turns out to be commutative;
hence $(\Sn_{2k},H_k)$ is called a \emph{Gelfand pair}
and $\setC[H_k \backslash \Sn_{2k} / H_k]$ its \emph{Hecke algebra}.
One can show that two permutations $\sigma$ and $\tau$ in $S_k$
are in the same double $H_k$-coset (\emph{i.e.}, there exist $h$ and $h'$ in $H_k$
such that $\sigma=h\, \tau \, h'$) if and only if 
the pairs $\big(S_\star,\sigma \cdot S_\star\big)$ and $\big(S_\star, \tau  \cdot S_\star\big)$
 have the same type.
  The type of $\big(S_\star,\sigma \cdot S_\star\big)$ is called \defin{coset-type} of $\sigma$.
  It plays an analogous role to the cycle-type for $\a=1$.

Given a Gelfand pair, one can define \emph{zonal spherical functions},
that play a similar role to that of normalized characters in representation theory.
In the case of the Gelfand pair $(\Sn_{2k},H_k)$, the zonal spherical functions
are usually denoted by $\macw$ and indexed by partition $\mu$ of $k$.
They are functions on the symmetric group $\Sn_{2k}$ that only depend on the coset-type of the permutation.
We will denote $\macw^{\mu}_\nu$ the value $\macw^{\mu}(\pi)$ of the zonal spherical function $\macw^{\mu}$
on any permutation $\pi$ of coset-type $\nu$.
We also use the notation $\macw^{\mu}_{(S_1,S_2)}$ for $\macw^{\mu}_\nu$ if $(S_1,S_2)$ has type $\nu$.

%
%
Zonal spherical functions of $(\Sn_{2k},H_k)$ are connected to Jack polynomials
by the following analogue of Frobenius formula (\cref{Subsec:Frobenius}):
$
z_{2\nu} \theta_\nu^{(2)}(\mu) = |H_k| \macw^{\mu}_\nu
$.
Recall that $|H_k|=2^k k!$ and 
$z_{2\nu}=z_\nu 2^{\ell(\nu)}$.
We can therefore rewrite the above equality as
\begin{equation}
    z_\nu 2^{\ell(\nu)} \theta_\nu^{(2)}(\mu) = 2^k k! \macw^{\mu}_\nu.
    \label{EqSphericalJack2Bis}
\end{equation}

\subsection{A family of polynomials indexed by pair-partitions}
As in \cref{SubsectN}, we introduce an auxiliary family of polynomials,
here indexed by three pair-partitions $S_0$, $S_1$ and $S_2$. Define
\begin{equation}
 N_{(S_0,S_1,S_2)}(\rvec,\pvec) =
\sum_{\substack{v: \U(S_0,S_2) \to [d] \\ w: \U(S_0,S_1) \to [d] \\ v, w \text{ compatible}}}
\left(\prod_{ B \in \U(S_0,S_2)} p_{v(B)}\right)
\left( \prod_{ B' \in \U(S_0,S_1)} r_{w(B')}\right).
\end{equation}


\begin{remark}
    As in the case $\a=1$ --- see \cref{RkMax} ---
    $N_{(S_0,S_1,S_2)}(\rvec,\pvec)$ could be equivalently defined
    using Stanley's original version of multirectangular coordinates
    and the max-compatibility condition.
    In particular, the function $N_{(S_0,S_1,S_2)}$ defined here
    corresponds to the function $N^{(1)}_{(S_0,S_1,S_2)}$ in \cite{NousZonal}
    --- see \cite[Lemma 3.9]{NousZonal}.
\end{remark}

\subsection{Combinatorial formulas}

Given a partition $\mu \vdash k$, 
let $S_1^\mu$ and $S_2^\mu$ denote arbitrarily chosen pair-partitions of $[2k]$
so that the type of $(S_1^\mu,S_2^\mu)$ is $\mu$.
We also choose an arbitrary set-partition $U_\mu$
so that the lengths of its part, sorted in decreasing order, are given by $2\mu \coloneqq (2\mu_1,2\mu_2,\dotsc)$.

\begin{proposition}\label{combFormulasInFerayNA2}
The polynomials $\Ch^{(2)}_\mu(\rvec^\pvec)$, $\zonalZ^\sh_\mu(\rvec^\pvec)$ and $\Ko^{(2)}_\mu(\rvec^\pvec)$
are expressed in terms of $N_{(S_0,S_1,S_2)}$ as follows:
\begin{align}
    \Ch^{(2)}_\mu(\rvec^\pvec) &= \frac{(-1)^k}{2^{l(\mu)}}\sum_{S_0 \in \PP{k}} (-2)^{|\U(S_0,S^\mu_1)|} N_{(S_0,S^\mu_1,S^\mu_2)}(\rvec,\pvec) \, ;
\label{EqChN2}\\
\zonalZ^\sh_\mu(\rvec^\pvec) &= \frac{(-1)^k k!}{(2k)!} \sum_{S_0,S_1,S_2 \in \PP{k}} \macw^{\mu}_{(S_1,S_2)} (-2)^{|\U(S_0,S_1)|} N_{(S_0,S_1,S_2)}(\rvec,\pvec) \, ;
\label{EqShZonalN} \\
\Ko^{(2)}_\mu(\rvec^\pvec) &= \frac{(-1)^k}{\prod_i (2\mu_i)!}
\left[ \sum_{ \substack{ S_0,S_1,S_2 \in \PP{k} \\ S_0,S_1 \le U_\mu }}          
(-2)^{|\U(S_0,S_1)|} N_{(S_0,S_1,S_2)}(\rvec,\pvec) \right].
     \label{EqKoN2}
\end{align}
\end{proposition}
\begin{proof}
The first statement is a reformulation of \cite[Thm. 1.6]{NousZonal}.

To prove the second statement, we use \cref{PropShJackOnCha} together with \eqref{EqChN2}:
\begin{align*}
\zonalZ^\sh_\mu(\rvec^\pvec)
&=
\sum_{\nu \vdash k}
2^{-(|\nu|-\length(\nu))} \theta_\nu^{(2)}(\mu)
\frac{(-1)^k}{2^{\length(\nu)}}
\sum_{S_0 \in \PP{k}} (-2)^{|\U(S_0,S_1^\nu)|} N_{(S_0,S_1^\nu,S_2^\nu)}(\rvec,\pvec) \\
&=
(-1)^k 2^{-k} \sum_{\nu \vdash k}
\theta_\nu^{(2)}(\mu)
\sum_{S_0 \in \PP{k}} (-2)^{|\U(S_0,S_1^\nu)|} N_{(S_0,S_1^\nu,S_2^\nu)}(\rvec,\pvec).
\end{align*}
From \cite[Lemma 2.4]{NousZonal}, for each partition $\nu \vdash k$,
there are exactly $(2k)!/(z_\nu \, 2^{\ell(\nu)})$ pairs $(S_1,S_2)$ of pair-partitions of $[2k]$
with type $\nu$.
Therefore,
\[
\zonalZ^\sh_\mu(\rvec^\pvec)
= (-1)^k 2^{-k} \sum_{S_0,S_1,S_2 \in \PP{k}}
\frac{z_\nu2^{\length(\nu)}}{(2k)!}
\theta_\nu^{(2)}(\mu)
(-2)^{|\U(S_0,S_1)|}N_{(S_0,S_1,S_2)}(\rvec,\pvec),
\]
where $\nu$ is the type of $(S_1,S_2)$.
Using \cref{EqSphericalJack2Bis},
the expression above further simplifies to:
\[
\zonalZ^\sh_\mu(\rvec^\pvec)
= \frac{(-1)^k k!}{(2k)!} \sum_{S_0,S_1,S_2 \in \PP{k}}
\macw^{\mu}_\nu
(-2)^{|\U(S_0,S_1)|}N_{(S_0,S_1,S_2)}(\rvec,\pvec).
\]

Let us now prove \cref{EqKoN2}.
Specializing \cref{PropKoaOnCha} for $\a=2$, we get
\begin{equation}
    \Ko^{(2)}_\mu(\rvec^\pvec) = \sum_{\nu \vdash k } \frac{ L_{\nu,\mu} }{z_\nu} \Ch^{(2)}_\nu(\lambda)  =
    \sum_{\nu \vdash k } \frac{ L_{\nu,\mu}}{z_\nu} 
\frac{(-1)^k}{2^{\length(\nu)}}
\sum_{S_0 \in \PP{k}} (-2)^{|\U(S_0,S_1^\nu)|} N_{(S_0,S_1^\nu,S_2^\nu)}(\rvec,\pvec). \\
\end{equation}
Using the same trick as above, we rewrite this as:
\begin{multline*}
    \Ko_\mu(\rvec^\pvec) = (-1)^k \sum_{S_0,S_1,S_2 \in \PP{k}} \frac{z_\nu 2^{\ell(\nu)}}{(2k)!}
    \frac{ L_{\nu,\mu}}{z_\nu 2^{\ell(\nu)}} \,(-2)^{|\U(S_0,S_1)|} N_{(S_0,S_1,S_2)}(\rvec,\pvec) \\
= \frac{(-1)^k}{(2k)!} \sum_{S_0,S_1,S_2 \in \PP{k}} L_{\nu,\mu} (-2)^{|\U(S_0,S_1)|} N_{(S_0,S_1,S_2)}(\rvec,\pvec),
\end{multline*}
where $\nu$ denotes the type of $(S_1,S_2)$.
From a straight-forward variant of \cref{LemLPerm}
$L_{\nu,\mu}$ is the number of functions 
$f:[2k] \to [\ell(\mu)]$, that are constant on blocks of $\U(S_0,S_1)$
and such that,
each number $i$ has exactly $2\mu_i$ pre-images by $f$.
Hence,
\[ \Ko^{(2)}_\mu(\rvec^\pvec) = 
\frac{(-1)^k}{(2k)!} \sum_{ \substack{f:[2k] \to [\ell(\mu)] \\ |f^{-1}(i)|=2\mu_i}}
\left[ \sum_{ \substack{ S_0,S_1,S_2 \in \PP{k} \\ \U(S_0,S_1) \le \Pi(f) }}
(-2)^{|\U(S_0,S_1)|} N_{(S_0,S_1,S_2)}(\rvec,\pvec) \right]. \]
One can easily prove that the expression in bracket does not depend on $f$.
Recall that the number of functions $f$ with $|f_0^{-1}(i)|=2\mu_i$ (for all $i\le \ell(\mu)$)
is $(2k)!/\prod_i (2\mu_i)!$.
Hence, if $f_0$ is such a function (arbitrarily chosen), we have:
\[
\Ko^{(2)}_\mu(\rvec^\pvec) =\frac{(-1)^k}{\prod_i (2\mu_i)!}
\left[ \sum_{ \substack{ S_0,S_1,S_2 \in \PP{k} \\ \U(S_0,S_1) \le \Pi(f_0) }}          
(-2)^{|\U(S_0,S_1)|} N_{(S_0,S_1,S_2)}(\rvec,\pvec) \right].
\]
But we can assume without loss of generality that the partition $\Pi(f_0)=U_\mu$.
Moreover, \hbox{$\U(S_0,S_1) \le U_\mu$} is by definition equivalent to $S_0 \le U_\mu$ and $S_1 \le U_\mu$.
This completes the proof of \cref{EqKoN2}.
\end{proof}

\subsection{Discussion}
As mentioned earlier, despite the combinatorial formulas given in \cref{combFormulasInFerayNA2},
we have not been able to prove \cref{ConjMain} for $\a=2$.
Let us explain briefly why.

The core of the proof of \cref{thm:mainshifted} is \cref{lem:nonnegsum};
what is before is mainly formal manipulations that are independent
of the specific structure for $\a=1$.
Most ingredients of the proof of \cref{lem:nonnegsum} have an analogue for $\a=2$:
\begin{itemize}
\item $X_\mu$ can be replaced by $X^{(2)}_{\mu} \coloneqq \sum_{\pi \in S_{2n}} \macw^\mu(\pi)\pi$,
which form a basis of orthogonal quasi-idempotents in $\setC[H_k \backslash \Sn_{2k} / H_k]$
\cite[Section VII,1]{Macdonald1995}.
\item The factor $(-2)^{|\U(S_0,S_1)|}$ in \cref{EqShZonalN} can be interpreted (up to a constant factor)
 as the zonal spherical function $\macw^{(1^k)}_{S_0,S_1}$ 
indexed by a one-column partition \cite[Section VII.2, Example 2.(b)]{Macdonald1995};
therefore $\varSigma_V$ has a natural analogue
$\varSigma^{(2)}_V :=   \sum_{\tau   \in S_{V}} \macw^{(1^k)}(\tau) \, \tau$.
\end{itemize}
These elements are quasi-idempotents and self-conjugate, as in the case $\a=1$.
Only the final manipulation of the proof of \cref{lem:nonnegsum}
cannot be performed for $\a=2$ since $X^{(2)}_{\mu}$ is \emph{not a central element}.

Computer exploration shows that the analogue of the quantity $B^{\mu}_{S,T}$ is sometimes negative.
Therefore a different approach has to be used to prove \cref{ConjMain} for $\a=2$.

\section{Positivity of  \texorpdfstring{$\Ko^{(\alpha)}_{(k)}$}{normalized alpha-Kostka} in the falling factorial basis}
\label{SectOnePart}

In this section, we prove combinatorially Theorem \ref{ThmOnePart}:
for each $k \ge 1$, the quantity $\Ko^{(\alpha)}_{(k)}(\rvec^\pvec)$
has nonnegative coefficients in the $\a$ falling factorial basis.
The proof is based on
a combinatorial formula for Jack polynomials,
given by Knop and Sahi in \cite{KnopSahiCombinatoricsJack}.

\subsection{The Knop--Sahi combinatorial formula}

Following Knop and Sahi, an \emph{admissible tableau} of shape $\la$ is a filling of the Young diagram $\la$
with positive integers such that
\begin{enumerate}
\item $T(i,j) \neq T(i', j-1)$ whenever $(i,j)$ is in $T$, $j>1$ and $i'<i$,
\item the same number does not appear twice in a column.
\end{enumerate}
As an example,
\[
\ytableausetup{mathmode,boxsize=1.2em}
\begin{yt}
1 & 3 & 2 & \mathbf{2} \\
4 & \mathbf{4} & 1 \\
6 & 2 \\
2 & 5
\end{yt} 
\]
is an admissible tableau (rows are numbered from top to bottom). 
A box $(i,j)$ in an admissible tableau $T$ is
\emph{critical} if $j>1$ and $T(i,j) = T(i,j-1)$. 
The two boxes with bold entries are critical in the example above.
The \defin{weight} of an admissible tableau is defined as
\[
d_T(\alpha) = \prod_{\square \in T \text{ critical}} 
\big[ \alpha(a_\lambda(\square) +1) + (l_\lambda(\square) +1) \big],
\]
where $a_\lambda(\square)$ and $l_\lambda(\square)$ are
the arm- and leg-lengths, as defined in \cref{SubsecPartitions}.

Knop and Sahi give a combinatorial formula for the Jack symmetric functions
as a weighted sum over admissible tableaux:
\begin{theorem}[F. Knop, S. Sahi, \cite{KnopSahiCombinatoricsJack}]
    For any Young diagram $\la$,
 \begin{equation}
  \jackJ^{(\alpha)}_\la = \sum_{T} d_T(\alpha) \prod_{s \in T} x_{T(s)}
 \end{equation}
where the (infinite) sum is taken over all admissible tableaux of shape $\la$.
\end{theorem}

By definition of $\Ko^{(\alpha)}_{(k)}$,
for $|\la|=n \ge k$, one has
\begin{equation}\label{eq:jackkostka}
      \Ko_{(k)}^{(\alpha)}(\lambda) = \frac{1}{(n-k)!} [m_{(k, 1^{n-k})}] J_{\lambda}^{(\alpha)}
\end{equation}
where $[m_{(k, 1^{n-k})}]F$ denotes the coefficient of $m_{(k, 1^{n-k})}$ in the symmetric function $F$.
Equivalently, this is the coefficient of any monomial whose exponent is a permutation
of $(k,1^{n-k})$, for instance $x_1^k x_2 \dots x_{n-k+1}$.
Therefore one has
\begin{equation}\label{eq:Kok_FromKS}
    \Ko_{(k)}^{(\alpha)}(\lambda) = \frac{1}{(n-k)!} \sum_{T} d_T(\alpha)
\end{equation}                     
where the sum is taken over all admissible tableaux of shape $\lambda$
and of type $(k,1^{n-k})$, that is
having $k$ entries equal to $1$, and one entry equal to each $i$ between $2$
and $n-k$.
\medskip

We now give an easy equivalent formulation of \cref{eq:Kok_FromKS},
more suitable for our purpose.

\begin{definition}
    \label{def:KSHT}
    A \emph{Knop--Sahi hook tableau} of shape $\lambda$ is the following data:
\begin{enumerate}
    \item \label{item:forbidden}
    a set of $k$ boxes, called \emph{marked boxes}, in different columns of $\lambda$, such that
no marked box is down-right to another marked box in two adjacent columns;
\item \label{item:critical}
    the right box in a pair of two adjacent marked boxes in a row is called \emph{critical};
\item \label{item:arrows} every critical box has 
    \begin{itemize}
        \item either a right arrow pointing at a box 
            which is \textbf{weakly to its right},
        \item or a down arrow pointing at a box
            which is \textbf{strictly below it},
        \item or no arrow at all.
    \end{itemize}
\item \label{item:weight} each critical box with a right-pointing arrow has weight $\alpha$, and all other boxes has weight $1$.
\end{enumerate}
The set of Knop--Sahi hook tableaux of shape $\la$ is denoted $\tabHook(\lambda,k)$.
The weight $w$ of a tableau in $\tabHook(\lambda,k)$
is the product of the weights of the boxes in the tableau.
\end{definition}

\begin{example}
A Knop-Sahi hook tableau of shape $\la=(9,9,7,5)$ with $k=8$ marked boxes is represented below.
Marked boxes are represented with the symbol $\ast$.
Five of these marked boxes are critical.
Instead of using arrows, a superscript number $j$ indicates an arrow pointing at
the box $j$ steps to the right, while a subscript number indicates a down-pointing arrow,
pointing at the box $j$ steps below.
As there are three right-pointing arrows (three superscript numbers),
the weight of this particular tableau is $\a^3$.
\[
\ytableausetup{mathmode,boxsize=1.2em}
\begin{yt}
\;\ast\;  &       &       &       &   & \;\ast\; & \;\ast^2\; & \;\ast^1\; & \;\ast_1\; \\
        &\;\ast\; &\;\ast^4\; & \;\ast_2\;&   &        &        &        &        \\
        &       &       &       &   &        &     \\
        &       &       &       &   \\
\end{yt}
\]
\label{Ex:KSHT}
\end{example}

\begin{proposition}
    \label{Prop:Kok_FromKS}
For any integer $k \ge 1$ and Young diagram $\la$, one has:
    \[ \Ko_{(k)}^{(\alpha)}(\lambda) = \sum_{T \in \tabHook(\lambda,k) } w(T). \]
\end{proposition}
\begin{proof}
    If $n<k$ both sides are trivially equal to $0$, so let us focus on $n \ge k$.

    We start from \cref{eq:Kok_FromKS}.
Consider an admissible tableau of shape $\la$ and type $(k,1^{n-k})$.
The $k$ boxes filled with $1$ satisfy Condition \ref{item:forbidden} of
\cref{def:KSHT}.

Moreover,
the conditions defining admissible tableaux only involve pair of boxes with the same value.
Hence, we can permute entries $2$, \dots, $n-k+1$ in an admissible tableau of type $(k,1^{n-k})$
and get another admissible tableau of type $(k,1^{n-k})$.
As the weight is also defined using only pairs of boxes with the same value,
this operation does not change the weight.

Finally, every placement of $k$ entries equal to $1$ in $\la$
respecting Condition \ref{item:forbidden}
of \cref{def:KSHT} can be completed in $(n-k)!$ ways to an
admissible tableaux of type $(k,1^{n-k})$,
each with the same weight.

Therefore, \cref{eq:Kok_FromKS} can be rewritten as follows:
\[ \Ko_{(k)}^{(\alpha)}(\lambda) = \sum_{T} d_T(\alpha), \]
where the sum runs over sets $T$ of $k$ boxes as in \cref{item:forbidden} 
of \cref{def:KSHT} and 
\[
d_T(\alpha) = \prod_{\square \in T \text{ critical}} 
\big[ \alpha(a_\lambda(\square) +1) + (l_\lambda(\square) +1) \big].
\]
\emph{Critical} is here defined as in \cref{item:critical} in \cref{def:KSHT}.
\medskip 

Now, \cref{item:arrows,item:weight} in \cref{def:KSHT}
ensure that the total weighted counts agree:
a critical box $\square$ has $a_\lambda(\square) +1$ possible 
right-pointing arrows each with weight $\a$ 
and $l_\lambda(\square)$ possible down-pointing arrows each with weight $1$.
To this should be added the possibility of having no arrow, in which
case the box also has weight $1$.
\end{proof}

We cannot prove \cref{ThmOnePart} directly 
from \cref{Prop:Kok_FromKS}.
In the next subsection, we shall give another equivalent 
combinatorial description of $\Ko_{(k)}^{(\alpha)}(\lambda)$.

\subsection{A new combinatorial formula for hook monomial coefficients}

We now give a weight-preserving bijection from $\tabHook(\lambda,k)$
to a second family of tableaux,
called \emph{permuted tableaux}.
The latter is easier to analyze with respect to multi-rectangular coordinates,
see \cref{SubsectKkFFPos}.

\subsubsection{Left-to-right minima}
Recall that $\Sn_j$ denotes the set of permutations of size $j$.
We shall here view permutations as words.
By definition,
the $i$-th entry $\pi_i$ of a permutation $\pi$ is
a \emph{left-to-right minimum} if, for any $i'<i$, one has $\pi_{i'}>\pi_i$.
In other terms, when we reach $\pi_i$, 
reading the permutation from left to right,
the value $\pi_i$ is the minimum
of the already read entries.
As an example,
$4\, 2\, 5\, 1\, 3$ has three left-to-right minima $\pi_1=4$,
$\pi_2=2$ and $\pi_4=1$.

The number of left-to-right minima of a permutation $\pi$ will be denoted $\lrmin(\pi)$.
All permutations have at least one left-to-right minimum, its first letter $\pi_1$.

Observe that each permutation $\pi$ of $j+1$ can be obtained uniquely from a permutation $\hat{\pi}$
of $j$ by inserting the letter $j+1$ in some position in the word of $\hat{\pi}$.
The number of left-to-right minima is easily tracked through this process:
\begin{itemize}
    \item if the \emph{new maximum} $j+1$ is inserted \emph{at the beginning}
        of the permutation,
        then $\lrmin(\pi) = 1 + \lrmin(\hat{\pi})$.
    \item otherwise, $\lrmin(\pi) = \lrmin(\hat{\pi})$. We will refer to this case
        by saying that the \emph{new maximum} $j+1$
        is inserted \emph{inside} the permutation $\hat{\pi}$.
\end{itemize}
An immediate consequence, obtained by induction,
is the following formula for the generating polynomial of $\lrmin$:
\begin{equation}
    \sum_{\pi \in \Sn_j} t^{\lrmin(\pi)} = t\, (t+1)\, \cdots\, (t+j-1),
    \label{eq:gen_pol_lrmin}
\end{equation}
or, equivalently,
\begin{equation}
    \sum_{\pi \in \Sn_j} t^{j-\lrmin(\pi)} = 
    (1+t)\, \cdots\, \big(1+(j-1)\, t\big).
    \label{eq:gen_pol_jlrmin}
\end{equation}

\subsubsection{Our new formula: first version}
\begin{definition}
Let $\tabPerm(\lambda,k)$ be the set of tableaux of shape $\lambda$ and
$k$ marked boxes, labeled with positive integers, with the following properties:
\begin{enumerate}
\item No two marked boxes appear the same column;
\item if a row $r$ contains $j$ marked boxes,
    then reading the labels from left to right give a permutation $\pi_r$ in $\perms_j$;
\item the weight of a row $r$ with label permutation $\pi_r$ in $\perms_j$
    is given by $\alpha^{j-\lrmin(\pi_r)}$,
where $\lrmin$ denotes the number of left-to-right minima.
\end{enumerate}
The weight $w(T)$ of a tableau $T$ in $\tabPerm(\lambda,k)$
is given by the product of the weights of the rows.
\end{definition}
\begin{example}
    The following is a permuted tableau of weight $\a^3$:
    indeed, the permutation in the first row, $2\, 1$, 
    and the permutation in the second row, $4\, 1\, 5\, 2\, 3$,
    have both two left-to-right minima,
    while the trivial permutation in the last row, $1$, has
    one left-to-right minimum.
\[ \ytableausetup{mathmode,boxsize=1.2em}
\begin{yt}
2  &       &       &       &   &  1 &   &        &   \\
        &4 & 1 &       &   &        &   5     &   2     &  3      \\
        &       &       &       &   &        &     \\
        &       &       &   1    &   \\
\end{yt}\]
\label{Ex:PermTab}
\end{example}

The purpose of the subsection is to prove the following new combinatorial formula
for $\Ko_{(k)}^{(\alpha)}(\lambda)$.
\begin{proposition}
    \label{Prop:NewCombFormula1}
For any integer $k \ge 1$ and Young diagram $\la$, one has:
\[
 \Ko_{(k)}^{(\alpha)}(\lambda) =  \sum_{T \in \tabPerm(\lambda,k) } w(T).
\]
\end{proposition}

\begin{proof}[Structure of the proof]
We will construct a weight-preserving bijection 
$\Psi: \tabHook(\la,k) \to \tabPerm(\la,k)$;
see \cref{lem:PsiPhiInverse,lem:Psi_WellDefined} below.
The proposition then follows from \cref{Prop:Kok_FromKS}.
\end{proof}

\subsubsection{The map $\Psi$}

We start with a Knop--Sahi hook tableau $T$.
To define its image by $\Psi$,
first replace the right-most marked box by a $1$.
Then we shall process marked boxes from right to left.
The box being processed is called the \emph{active box}.
At each stage, but the last, 
the next active 
box gets a label, which is a positive integer.
Therefore, already processed boxes, 
as well as the active box being processed, have a label.

We shall often assign to a box \emph{the new maximal label in its row}.
This simply means that we give it label $m = 1 + \max(b_1,\dots,b_l)$,
where $b_1,\dots,b_l$ are the labels of already labeled boxes in the considered row.
If the considered row has no labeled boxes yet, take by convention $m=1$.

We now explain how to process boxes in a case-by-case fashion:
\begin{description}
    \item[Case (N)] \textbf{The active box is not critical or has no arrows.}
Do nothing with the active box.
If the active box is not the left-most marked box of the tableau,
\emph{i.e.}, if this is not the last step,
assign to the next active box the maximal label in its row.

This is represented in the Figure below.
The dots indicate that there may be some columns without marked boxes
between the current active box and the next one.
Note also that the next active box can be above, below, or in the same row
as the current active box --- this is not relevant here.
\medskip
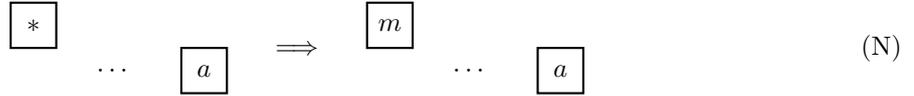
\begin{equation}\label{eq:ksbij-N}
    \begin{array}{c} \begin{tikzpicture}[->,thick,baseline=-0.25ex,
    block/.style={rectangle, draw=black,minimum size=6mm},
every loop/.style={},node distance=6mm]
\node[block] (X) {$\ast$};
\node [right=18mm of X] (B) { };
\node[block] [below of=B] (A) {$a$};
\node [left=5mm of A] {$\cdots$};
\end{tikzpicture}\end{array}
\quad
\Longrightarrow
\quad
  \begin{array}{c}
\begin{tikzpicture}[->,thick,baseline=-0.25ex,
    block/.style={rectangle, draw=black,minimum size=6mm},
every loop/.style={},node distance=6mm]
\node[block] (X) {$m$};
\node [right=18mm of X] (B) { };
\node[block] [below of=B] (A) {$a$};
\node [left=5mm of A] {$\cdots$};
\end{tikzpicture} \end{array}\tag{N}
\end{equation}
\item[Case (D)] \textbf{The active box is critical and has a down-pointing arrow.}
    First give the label $a$ of the active box to the next active box
    (which is immediately on its left, as the active box is critical).
    Then unmark the active box and mark instead the pointed box;
    assign to it the maximal label in its row.
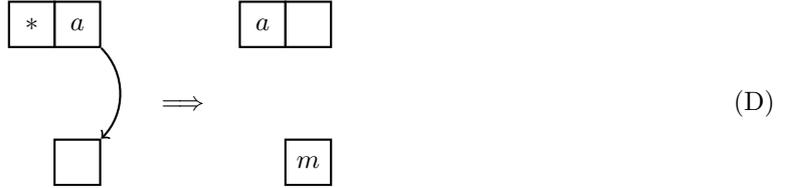
\begin{equation}\label{eq:ksbij-down}
\begin{tikzpicture}[->,thick,baseline=-7.5ex,
    block/.style={rectangle, draw=black,minimum size=6mm},
every loop/.style={},node distance=6mm]
\node[block] (X) {$\ast$};
\node[block] [right of=X] (A) {$a$};
\node[block] [below=12mm of A] (B) {$ $};
\path[every node/.style={font=\sffamily\small}]
(A) edge[bend left=45] 
(B);
\end{tikzpicture}
\quad
\Longrightarrow
\quad
\begin{tikzpicture}[->,thick,baseline=-7.5ex,
    block/.style={rectangle, draw=black,minimum size=6mm},
every loop/.style={},node distance=6mm]
\node[block] (X) {$a$};
\node[block] [right of=X] (A) {$ $};
\node[block] [below=12mm of A] (B) {$m$};
\end{tikzpicture}\tag{D}
\end{equation}

\item[Case (R)] \textbf{The active box is critical and has a right-pointing arrow.}

    \emph{Subcase (Re), the column containing the pointed box is empty.}
    Proceed exactly as in case $(D)$.
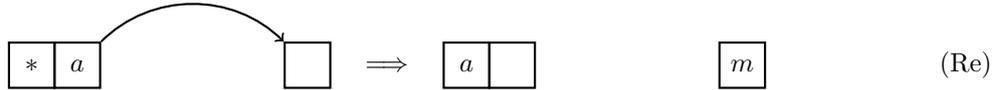
\begin{equation}\label{eq:ksbij-right}
\begin{tikzpicture}[->,thick,baseline=-0.5ex,
    block/.style={rectangle, draw=black,minimum size=6mm},
every loop/.style={},node distance=6mm]
\node[block] (X) {$\ast$};
\node[block] [right of=X] (A) {$a$};
 \node[block] [right=24mm of A] (B) {$ $};
\path[every node/.style={font=\sffamily\small}]
(A) edge[bend left=45] (B);
\end{tikzpicture}
\quad
\Longrightarrow
\quad
\begin{tikzpicture}[->,thick,baseline=-0.5ex,
    block/.style={rectangle, draw=black,minimum size=6mm},
every loop/.style={},node distance=6mm]
\node[block] (X) {$a$};
\node[block] [right of=X] (A) {$ $};
\node[block] [right=24mm of A] (B) {$m$};
\end{tikzpicture}\tag{Re}
\end{equation}

\noindent
\emph{Subcase (Ra), there is already a marked box $b_1$ \textbf{weakly above} the pointed box
in the same column.}
Do as in case $(Re)$.
In addition, we perform a shift of marked boxes in the same row as $b$ as displayed below.
This operation ensures that we never have two marked boxes in the same column.
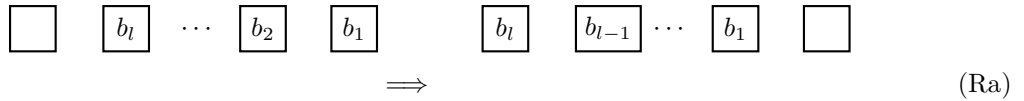
\begin{equation}\label{eq:ksbij-rightA}
\begin{tikzpicture}[->,thick,baseline=6.5ex,
    block/.style={rectangle, draw=black,minimum size=6mm},
every loop/.style={},node distance=6mm]
\node[block] (X) {$\ast$};
\node[block] [right of=X] (A) {$a$};
\node[block] [above=12mm of A] (DD) {$ $};
\node[block] [right=36mm of A] (G) {$ $};
\node[block] [right=6mm of DD] (B3) {$b_l$};
\node [right=15mm of DD] (BB) {$\cdots$};
\node[block] [right=24mm of DD] (B2) {$b_2$};
\node[block] [right=36mm of DD] (B1) {$b_1$};
\path[every node/.style={font=\sffamily\small}]
(A) edge[bend right=-20] (G);
\end{tikzpicture}
\Longrightarrow
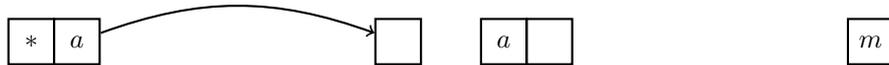
\begin{tikzpicture}[->,thick,baseline=6.5ex,
    block/.style={rectangle, draw=black,minimum size=6mm},
every loop/.style={},node distance=6mm]
\node[block] (X) {$a$};
\node[block] [right of=X] (A) {$ $};
\node[block] [above=12mm of A] (DD) {$b_l$};
\node[block] [right=36mm of A] (G) {$m$};
\node[block] [right=6mm of DD] (B3) {$b_{l-1}$};
\node [right=15mm of DD] (BB) {$\cdots$};
\node[block] [right=24mm of DD] (B2) {$b_1$};
\node[block] [right=36mm of DD] (B1) {$ $};
\end{tikzpicture}\tag{Ra}
\end{equation}

\noindent
\emph{Subcase (Rb), there is already a marked box $b_1$ \textbf{strictly below} the pointed box
in the same column.}
We do as in Subcase $(Ra)$ above, except that the new marked box is not the 
pointed box, but the former box $b_1$; see Figure below.
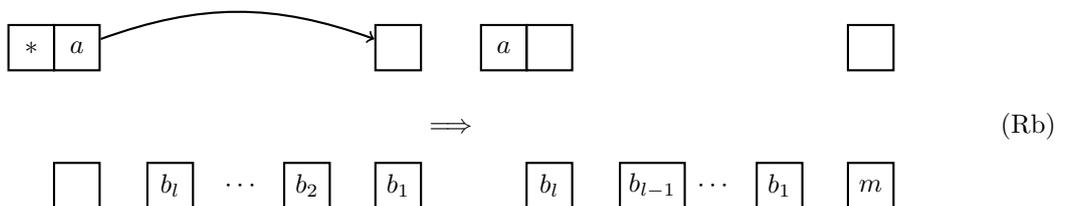
\begin{equation}\label{eq:ksbij-rightB}
\begin{tikzpicture}[->,thick,baseline=-7.5ex,
    block/.style={rectangle, draw=black,minimum size=6mm},
every loop/.style={},node distance=6mm]
\node[block] (X) {$\ast$};
\node[block] [right of=X] (A) {$a$};
\node[block] [below=12mm of A] (DD) {$ $};
\node[block] [right=36mm of A] (G) {$ $};
\node[block] [right=6mm of DD] (B3) {$b_l$};
\node [right=15mm of DD] (BB) {$\cdots$};
\node[block] [right=24mm of DD] (B2) {$b_2$};
\node[block] [right=36mm of DD] (B1) {$b_1$};
\path[every node/.style={font=\sffamily\small}]
(A) edge[bend left=20] (G);
\end{tikzpicture}
\Longrightarrow
\begin{tikzpicture}[->,thick,baseline=-7.5ex,
    block/.style={rectangle, draw=black,minimum size=6mm},
every loop/.style={},node distance=6mm]
\node[block] (X) {$a$};
\node[block] [right of=X] (A) {$ $};
\node[block] [below=12mm of A] (DD) {$b_l$};
\node[block] [right=36mm of A] (G) {$ $};
\node[block] [right=6mm of DD] (B3) {$b_{l-1}$};
\node [right=15mm of DD] (BB) {$\cdots$};
\node[block] [right=24mm of DD] (B2) {$b_1$};
\node[block] [right=36mm of DD] (B1) {$m$};
\end{tikzpicture}\tag{Rb}
\end{equation}

%
\emph{
The special case where the active box has a right arrow pointing to itself
is considered as a degenerate case of $(Ra)$.
}
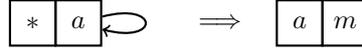
\[
\begin{tikzpicture}[->,thick,baseline=-0.5ex,
    block/.style={rectangle, draw=black,minimum size=6mm},
every loop/.style={},node distance=6mm]
\node[block] (X) {$\ast$};
\node[block] [right of=X] (A) {$a$};
\path[every node/.style={font=\sffamily\small}]
(A)  edge[loop right,looseness=12] node  { } (A);
\end{tikzpicture}
\quad
\Longrightarrow
\quad
\begin{tikzpicture}[->,thick,baseline=-0.5ex,
    block/.style={rectangle, draw=black,minimum size=6mm},
every loop/.style={},node distance=6mm]
\node[block] (X) {$a$};
\node[block] [right of=X] (A) {$m$};
\end{tikzpicture}
\]
\end{description}

When all marked boxes have been processed,
we get a tableau $T'$ with labeled marked boxes.
We set $\Psi(T)=T'$.

\begin{example}
This example illustrates the map described above.
We start from the Knop-Sahi hook tableau $T$ from \cref{Ex:KSHT}.
The first step is to assign label 1 to the right-most box and we get:
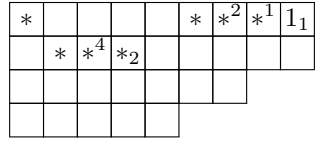
\[
\ytableausetup{mathmode,boxsize=1.2em}
\begin{yt}
\;\ast\;  &       &       &       &   & \;\ast\; & \;\ast^2\; & \;\ast^1\; & \;1_1\; \\
        &\;\ast\; &\;\ast^4\; & \;\ast_2\;&   &        &        &        &        \\
        &       &       &       &   &        &     \\
        &       &       &       &   \\
\end{yt}
\]
Since this box, which is the first active box,
has a down pointing arrow, we apply rule \eqref{eq:ksbij-down}.
The label $1$ is transfered to the box immediately at his left,
and instead of the active box, we mark the pointed box, 
in the same column, one row below. 
This box gets  the new maximal label in its row, $1$:
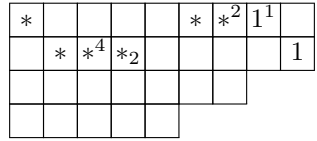
\[
\ytableausetup{mathmode,boxsize=1.2em}
\begin{yt}
\;\ast\;  &       &       &       &   & \;\ast\; & \;\ast^2\; & {1}^1 &   \\
        &\;\ast\; &\;\ast^4\; & \;\ast_2\;&   &        &        &        &  1      \\
        &       &       &       &   &        &     \\
        &       &       &       &   \\
\end{yt}
\]
In the next step, the new active box has a right-pointing arrow, one step to the right.
The column the arrow points already contains a marked box below the pointed box,
so we apply rule \eqref{eq:ksbij-rightB}.
In the step after that, we again apply rule \eqref{eq:ksbij-rightB}.
These two steps are represented below.
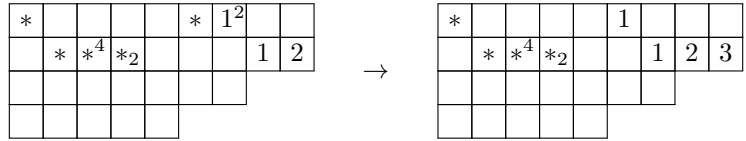
\[
\ytableausetup{mathmode,boxsize=1.2em}
\begin{yt}
\;\ast\;  &       &       &       &   & \;\ast\; & \;{1}^2 &        &   \\
        &\;\ast\; &\;\ast^4\; & \;\ast_2\;&   &        &        &   1     &  2      \\
        &       &       &       &   &        &     \\
        &       &       &       &   \\
\end{yt}
\quad \to \quad
\ytableausetup{mathmode,boxsize=1.2em}
\begin{yt}
\;\ast\;  &       &       &       &   &  1 &   &        &   \\
        &\;\ast\; &\;\ast^4\; & \;\ast_2\;&   &        &   1     &   2     &  3      \\
        &       &       &       &   &        &     \\
        &       &       &       &   \\
\end{yt}
\]
The active box is now the left-most $1$, which is not a critical box.
We apply rule \eqref{eq:ksbij-N}
and assign to the next active box the new maximal value
in its row, that is $4$.
This new active box has a down-pointing arrow, so that we apply rule \eqref{eq:ksbij-down}.
These two steps are represented below.
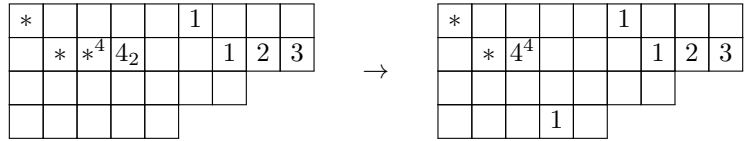
\[
\ytableausetup{mathmode,boxsize=1.2em}
\begin{yt}
\;\ast\;  &       &       &       &   &  1 &   &        &   \\
        &\;\ast\; &\;\ast^4\; & \;4_2\;&   &        &   1     &   2     &  3      \\
        &       &       &       &   &        &     \\
        &       &       &       &   \\
\end{yt}
\quad \to \quad
 \ytableausetup{mathmode,boxsize=1.2em}
\begin{yt}
\;\ast\;  &       &       &       &   &  1 &   &        &   \\
        &\;\ast\; & {4}^4 &       &   &        &   1     &   2     &  3      \\
        &       &       &       &   &        &     \\
        &       &       &   1    &   \\
\end{yt}
\]
In the next step, the active box has a right-arrow pointing
to a column $4$ steps to the right.
The pointed box is already occupied,
so that we apply rule \eqref{eq:ksbij-rightA}.
In the next step, the active box is not critical,
so that we give to the future active box the new maximum label in its row.
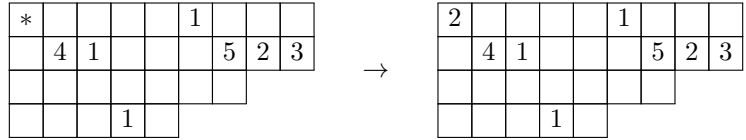
\[
 \ytableausetup{mathmode,boxsize=1.2em}
\begin{yt}
\;\ast\;  &       &       &       &   &  1 &   &        &   \\
        &4 & 1 &       &   &        &   5     &   2     &  3      \\
        &       &       &       &   &        &     \\
        &       &       &   1    &   \\
\end{yt}
\quad \to \quad
 \ytableausetup{mathmode,boxsize=1.2em}
\begin{yt}
2  &       &       &       &   &  1 &   &        &   \\
        &4 & 1 &       &   &        &   5     &   2     &  3      \\
        &       &       &       &   &        &     \\
        &       &       &   1    &   \\
\end{yt}
\]
The left-most marked box is not critical --- it never is by definition ---
and we do nothing.
Finally we get the permuted tableau from \cref{Ex:PermTab}.
\end{example}

\begin{lemma}
    Let $T$ be in $\tabHook(\la,k)$.
    The output of the above-described function $\Psi(T)$
    is a permuted tableau in $\tabPerm(\la,k)$.
    Moreover, $\Psi$ is weight-preserving.
    \label{lem:Psi_WellDefined}
\end{lemma}
\begin{proof}
    It is straightforward to check by a case-by-case analysis
    that the following is true at each stage of the construction.
    \begin{itemize}
        \item Marked boxes are in different columns.
        \item The labeled marked boxes in a given row
            form a permutation. Indeed, when we label a new box
            in a row, we always assign it the $1 + \max(b_1,\dots,b_l)$,
            where $b_1$, \dots, $b_l$ are the labels already present in this row.
            This clearly preserves the fact that the labels form a permutation.
        \item Neither the shape, nor the total number of marked boxes are changed.
    \end{itemize}

    The weight-preserving property is also obtained by a case-by-case analysis:
    \begin{itemize}
        \item In cases \eqref{eq:ksbij-N}
            and \eqref{eq:ksbij-down}, we add a new maximum 
            \emph{at the beginning} of the permutation $\pi$ corresponding to some row.
            This does not change the statistics $j-\lrmin(\pi)$:
            the number of left-to-right minima and the size of the permutation
            both increase by $1$.
            This fits with boxes without arrows (either critical or non-critical)
            and critical boxes
            with down-pointing arrows having weight $1$.
        \item In cases \eqref{eq:ksbij-right},
            \eqref{eq:ksbij-rightA} and \eqref{eq:ksbij-rightB},
            we add a new maximum
            \emph{inside} the permutation $\pi$ corresponding to some row.
            Then the statistics $j-\lrmin(\pi)$ increases by $1$:
            the size increases, but not the number of left-to-right minima.
            This fits with critical boxes with right-pointing arrows having weight $\a$.
            \qedhere
    \end{itemize}
\end{proof}

\subsubsection{The reverse map $\Phi$}
We start from a permuted tableaux $T'$.
To define its image $\Phi(T')$,
we shall process marked boxes from left to right.
As before, the box being processed is called \emph{active}.
At each stage, marked boxes on the right of the active box
have labels, while marked boxes on the left are unlabeled.
Already processed critical boxes might have arrows.

As above, we describe our procedure in a case-by-case fashion.
\begin{description}
    \item[Case (M)] \textbf{the active box has the maximal label in its row and
        there is no marked box in the column immediately to the right of the active box
        strictly below the active box.}
        When we apply $\Psi$,
        this situation may only occur after a step \eqref{eq:ksbij-N}.
        To reverse this step, one has just to remove the label of the active box.
\begin{equation}\label{eq:ksbij-M}
    \begin{array}{c}\begin{tikzpicture}[->,thick,
    block/.style={rectangle, draw=black,minimum size=6mm},
every loop/.style={},node distance=6mm]
\node[block] (X) {$m$};
\end{tikzpicture}\end{array}
\quad
\Longrightarrow
\quad
\begin{array}{c}
\begin{tikzpicture}[->,thick,
    block/.style={rectangle, draw=black,minimum size=6mm},
every loop/.style={},node distance=6mm]
\node[block] (X) {$\ast$};
\end{tikzpicture}\end{array}
\tag{M}
\end{equation}
A marked box in the column immediately to the right of the active box
strictly below the active box 
would correspond to
 the forbidden pattern in Knop--Sahi hook tableau,
 \cref{item:forbidden} of \cref{def:KSHT}
 (which is allowed in permuted tableaux).
This explains why this case has to be treated separately;
see case (B) below.

    \item[Case (E)] \textbf{
The label of the active box is not the maximum in its row
and the column immediately to the right of the active box is empty.}
        When we apply $\Psi$,
        this situation may only occur after a step \eqref{eq:ksbij-right}.
        To reverse this step, we should do the following.

Call $a$ the label of the active box. 
Look for the box $\Box_m$ containing the maximal label in the row of the active box.
We unmark this box, and replace it by a new marked box immediately to the right of the active box.
This new marked box gets label $a$ --- the active box is now unlabeled ---
and an arrow pointing to $\Box_m$. \vspace{-5mm}
\begin{equation}\label{eq:ksbij-E}
\begin{tikzpicture}[->,thick,baseline=-0.5ex,
    block/.style={rectangle, draw=black,minimum size=6mm},
every loop/.style={},node distance=6mm]
\node[block] (X) {$a$};
\node[block] [right of=X] (A) {$ $};
\node[block] [right=24mm of A] (B) {$m$};
\end{tikzpicture}
\quad
\Longrightarrow
\quad
\begin{tikzpicture}[->,thick,baseline=-0.5ex,
    block/.style={rectangle, draw=black,minimum size=6mm},
every loop/.style={},node distance=6mm]
\node[block] (X) {$\ast$};
\node[block] [right of=X] (A) {$a$};
 \node[block] [right=24mm of A] (B) {$ $};
\path[every node/.style={font=\sffamily\small}]
(A) edge[bend left=45] (B);
\end{tikzpicture}\tag{E}
\end{equation}

    \item[Case (A)] \textbf{
The label of the active box is not the maximum in its row
and there is a marked box $\Box_b$ in the column immediately
to the right of the active box weakly above the active box.}
        When we apply $\Psi$,
        this situation may only occur after a step \eqref{eq:ksbij-rightA}.
        To reverse this step, we should do the following.

Do as in case $(E)$. 
In addition, we perform a shift of the marked boxes in the row
as $\Box_b$ as in the following picture.
\begin{equation}
\begin{tikzpicture}[->,thick,baseline=6.5ex,
    block/.style={rectangle, draw=black,minimum size=6mm},
every loop/.style={},node distance=6mm]
\node[block] (X) {$a$};
\node[block] [right of=X] (A) {$ $};
\node[block] [above=12mm of A] (DD) {$b_l$};
\node[block] [right=36mm of A] (G) {$m$};
\node[block] [right=6mm of DD] (B3) {$b_{l-1}$};
\node [right=15mm of DD] (BB) {$\cdots$};
\node[block] [right=24mm of DD] (B2) {$b_1$};
\node[block] [right=36mm of DD] (B1) {$ $};
\end{tikzpicture}
\Longrightarrow
\begin{tikzpicture}[->,thick,baseline=6.5ex,
    block/.style={rectangle, draw=black,minimum size=6mm},
every loop/.style={},node distance=6mm]
\node[block] (X) {$\ast$};
\node[block] [right of=X] (A) {$a$};
\node[block] [above=12mm of A] (DD) {$ $};
\node[block] [right=36mm of A] (G) {$ $};
\node[block] [right=6mm of DD] (B3) {$b_l$};
\node [right=15mm of DD] (BB) {$\cdots$};
\node[block] [right=24mm of DD] (B2) {$b_2$};
\node[block] [right=36mm of DD] (B1) {$b_1$};
\path[every node/.style={font=\sffamily\small}]
(A) edge[bend left=-20] (G);
\end{tikzpicture}\tag{A}
\label{eq:ksbij-A}
\end{equation}

 \item[Case (B)] \textbf{there is a marked box $\Box$ in the column immediately
     to the right of the active box strictly below the active box.}

 \emph{Subcase (Bm): the label of this marked box $\Box_m$ is the maximum $m$
 in its row.}
        When we apply $\Psi$,
        this situation may only occur after a step \eqref{eq:ksbij-down}.
        To reverse this step, we should do the following.

 Unmark box $m$ and mark instead the box immediately to the right of the active box.
 Assign to this new marked box an arrow pointing to $\Box_m$ 
 and give it also the label of the active box --
 the active box is now unlabeled.
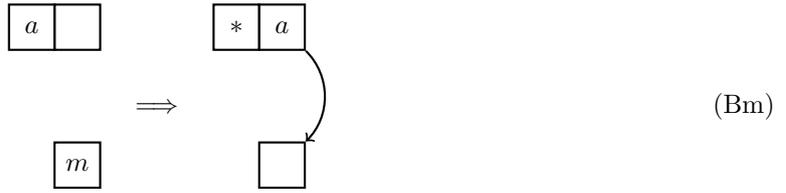
\begin{equation}\label{eq:ksbij-Bm}
\begin{tikzpicture}[->,thick,baseline=-7.5ex,
    block/.style={rectangle, draw=black,minimum size=6mm},
every loop/.style={},node distance=6mm]
\node[block] (X) {$a$};
\node[block] [right of=X] (A) {$ $};
\node[block] [below=12mm of A] (B) {$m$};
\end{tikzpicture}
\quad
\Longrightarrow
\quad
\begin{tikzpicture}[->,thick,baseline=-7.5ex,
    block/.style={rectangle, draw=black,minimum size=6mm},
every loop/.style={},node distance=6mm]
\node[block] (X) {$\ast$};
\node[block] [right of=X] (A) {$a$};
\node[block] [below=12mm of A] (B) {$ $};
\path[every node/.style={font=\sffamily\small}]
(A) edge[bend left=45] 
(B);
\end{tikzpicture}
\tag{Bm}
\end{equation}

\emph{Subcase (Bn): the label of this marked box $\Box_b$ is not the maximum $m$
 in its row.}
        When we apply $\Psi$,
        this situation may only occur after a step \eqref{eq:ksbij-rightB}.
        To reverse this step, we should do the following.

Look for the box $\Box_m$ with maximal label $m$ in the row of $\Box_b$.
Unmark this box and mark instead the box immediately to the right of the active box.
Assign to this new marked box a right arrow pointing to the column of $\Box_m$,
and give it the label $a$ of the active box -- the active box is now unlabeled.
Finally, perform a shift as in Case (A) to avoid having two boxes in the same column.
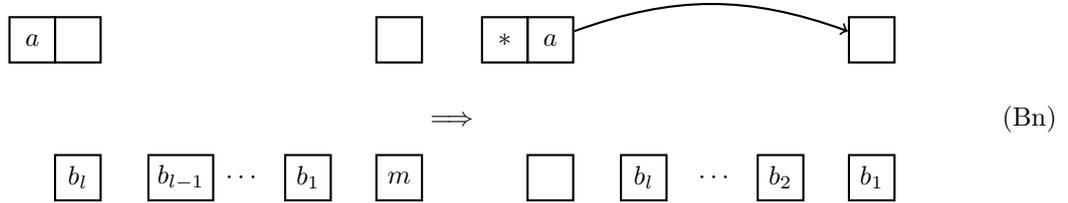
\begin{equation}\label{eq:ksbij-Bn}
\begin{tikzpicture}[->,thick,baseline=-7.5ex,
    block/.style={rectangle, draw=black,minimum size=6mm},
every loop/.style={},node distance=6mm]
\node[block] (X) {$a$};
\node[block] [right of=X] (A) {$ $};
\node[block] [below=12mm of A] (DD) {$b_l$};
\node[block] [right=36mm of A] (G) {$ $};
\node[block] [right=6mm of DD] (B3) {$b_{l-1}$};
\node [right=15mm of DD] (BB) {$\cdots$};
\node[block] [right=24mm of DD] (B2) {$b_1$};
\node[block] [right=36mm of DD] (B1) {$m$};
\end{tikzpicture}
\Longrightarrow
\begin{tikzpicture}[->,thick,baseline=-7.5ex,
    block/.style={rectangle, draw=black,minimum size=6mm},
every loop/.style={},node distance=6mm]
\node[block] (X) {$\ast$};
\node[block] [right of=X] (A) {$a$};
\node[block] [below=12mm of A] (DD) {$ $};
\node[block] [right=36mm of A] (G) {$ $};
\node[block] [right=6mm of DD] (B3) {$b_l$};
\node [right=15mm of DD] (BB) {$\cdots$};
\node[block] [right=24mm of DD] (B2) {$b_2$};
\node[block] [right=36mm of DD] (B1) {$b_1$};
\path[every node/.style={font=\sffamily\small}]
(A) edge[bend left=20] (G);
\end{tikzpicture}\tag{Bn}
\end{equation}
\end{description}
When all marked boxes have been processed,
we get a tableau $T$ with marked boxes and arrows pointing from critical boxes.
We set $\Phi(T')=T$.
\begin{example}
Here is an example of the inverse mapping.
Start from the following tableau
\[
T'=\ytableausetup{mathmode,boxsize=1.2em}
\begin{yt}
 \;& 2 &  &  &  &  &  &  & 3 & & 1 \\
 2 &   & &1 &  &  &  & 3&   &4&   \\
 & & 2 & & 3 & & 1
\end{yt}.
\]
The left-most marked box is the first active box.
As it is not the maximal in its row and the box at its upper right is marked,
we apply rule $(A)$.
Note that the maximum label in its row is $4$. We get
\[
\ytableausetup{mathmode,boxsize=1.2em}
\begin{yt}
 \;&   &  &  &  &  &  &        &  2 & 3 & 1 \\
 \ast & 2^8 & &1 &  &  &  & 3     &   & &   \\
 & & 2 & & 3 & & 1
\end{yt},
\]
where the superscript $8$ represents an arrow pointing $8$ boxes to the right,
that is to the box which used to contain $4$.

The new active box is the $2$ in the second column.
There is a marked box immediately to its right and below it,
so that we are in case $(B)$. 
As this box is not the maximum in its row, we apply rule $(Bn)$.
\[
\ytableausetup{mathmode,boxsize=1.2em}
\begin{yt}
 \;&   &  &  &  &  &  &        &  2 & 3 & 1 \\
 \ast & \ast^8 & 2^2 &1 &  &  &  & 3     &   & &   \\
 & & & & 2 & & 1
\end{yt},
\]
where the superscript $2$ indicates a right arrow pointing two boxes to the right,
which is the column that used to contain the maximum of the third row.

We proceed further and apply rules $(A)$ and $(Bm)$ to the boxes in the third and fourth columns,
respectively. We get:
\[
\ytableausetup{mathmode,boxsize=1.2em}
\begin{yt}
 \;&   &  &  &  &  &  &        &  2 & 3 & 1 \\
 \ast & \;\ast^8\; & \ast^2 &2^4 &  &  &  & 1     &   & &   \\
 & & & & 2 & & 1
\end{yt}
\quad \to \quad
\begin{yt}
 \;&   &  &  &  &  &  &        &  2 & 3 & 1 \\
 \;\ast & \;\ast^8\; & \ast^2 &\ast^4 &2_1  &  &  & 1     &   & &   \\
 & & & & & & 1
\end{yt}.
\]
The next active box (in the fifth column) has the maximal label in its row,
so that we just remove its label, according to rule $(M)$.  
After this operation, the same holds for boxes in seventh and eighth column,
so we also remove their label (see left figure below).
We now perform rule $(A)$ to the box in the ninth column.
The maximum in its row is $3$. We get
\[
\ytableausetup{mathmode,boxsize=1.2em}
\begin{yt}
 \;&   &  &  &  &  &  &        &  2 & 3 & 1 \\
 \;\ast & \;\ast^8\; & \ast^2 &\ast^4 &\ast_1  &  &  & \ast     &   & &   \\
 & & & & & & \ast
\end{yt}
\quad \to \quad
\begin{yt}
 \;&   &  &  &  &  &  &        &  \ast & 2^0 & 1 \\
 \;\ast & \;\ast^8\; & \ast^2 &\ast^4 &\ast_1  &  &  & \ast     &   & &   \\
 & & & & & & \ast
\end{yt}
\]
The superscript $0$ indicates a right arrow pointing to the box itself.
This is a degenerate case of rule $(A)$, where the maximum in the row of the active box
happens to be in the box immediately to its right.

The last two steps consist in applying twice rule $(M)$,
that is in removing the last two labels.
We get
\[
\Phi(T')=\ytableausetup{mathmode,boxsize=1.2em}
\begin{yt}
 \;&   &  &  &  &  &  &        &  \ast & \ast^0 & \ast \\
 \;\ast & \;\ast^8\; & \ast^2 &\ast^4 &\ast_1  &  &  & \ast     &   & &   \\
 & & & & & & \ast
\end{yt}.
\]
\end{example}

%

\begin{lemma}
    Let $T'$ be in $\tabPerm(\la,k)$.
    Then $\Phi(T')$ is an element of $\tabHook(\la,k)$.
    Moreover, $\Phi$ is the compositional inverse of $\Psi$.
    \label{lem:PsiPhiInverse}
\end{lemma}
\begin{proof}
    It is straightforward to check that, along the procedure,
    we never have two marked boxes in the same column
    or two unlabeled marked boxes forming the forbidden
    pattern of Knop-Sahi hook tableaux.
    Note also that we only assign arrows to critical boxes,
    pointing either to a box weakly on its right in the same row
    or strictly below in the same column.
    Moreover, neither the shape, nor the number of marked boxes 
    is modified.
    This implies that the output $\Phi(T')$ is an element of $\tabHook(\la,k)$.

    Now, $\Phi$ is the compositional inverse of $\Psi$
    since, by construction,
    every step of $\Phi$ reverses a step of $\Psi$:
    \eqref{eq:ksbij-M} for \eqref{eq:ksbij-N},
    \eqref{eq:ksbij-E} for \eqref{eq:ksbij-right},
    \eqref{eq:ksbij-A} for \eqref{eq:ksbij-rightA},
    \eqref{eq:ksbij-Bm} for \eqref{eq:ksbij-down},
    and \eqref{eq:ksbij-Bn} for \eqref{eq:ksbij-rightB}.
\end{proof}

\subsubsection{Our new formula: second version}
Here we view a Young diagram $\la$ as a set of boxes.
In particular, $A \subseteq \la$ means that we consider a subset
of the set of boxes of $\la$.
Such a subset is called \emph{column-distinct} if no two elements of $A$
are in the same column.
Rows of the diagram $\la$ will also be considered
as subsets of the set of boxes of $\la$.

With this viewpoint, \cref{Prop:NewCombFormula1} can be rewritten as follows.
\begin{theorem}
For any integer $k \ge 1$ and Young diagram $\la$, one has:
\[
\tfrac{1}{k!} \shJack_{(k)}(\lambda)=\Ko_{(k)}^{(\alpha)}(\lambda) =  \sum_{\substack{ A \subseteq \la, \ |A|=k \\ \text{column-dinstinct}}}
\left(
\prod_{\substack{R \text{ row} \\ \text{of }\la}} P_{|R \cap A|}(\a)
\right),
\]
where, for $i \ge 0$, we set $P_i(\a) = \prod_{j=0}^{i-1} (1 + j \, \a)$.
    \label{thm:NewCombFormula2}
\end{theorem}
\begin{proof}
  Recall that the first equality has been proved in \cref{Subsect:change_bases},
     so let us focus on the second.

    To a permuted tableau $T \in \tabPerm(\la,k)$,
    we can associate its set $M(T)$ of marked boxes:
    this is a column-distinct subset of cardinality $k$
    of the set of boxes of $\la$.

    Conversely each column-distinct subset $A$ of cardinality $k$ of $\la$
    is obtained from $\prod_{R} |R \cap A|!$ permuted tableaux,
    where the product is taken over rows of $\la$.
    Indeed to obtain such a tableau, one has to choose for each row $R$
    a permutation of size $|R \cap A|$, which gives the labels
    of marked boxes in this row.

    From the definition of the weight of permuted tableaux and
    \cref{eq:gen_pol_jlrmin}, we get that, for any
    column-distinct subset $A$ of cardinality $k$ of $\la$,
    \[\sum_{ \substack{ T \in \tabPerm(\la,k) \\ M(T)=A }} w(T)
    =\prod_{\substack{ R \text{ row} \\ \text{of }\la }} P_{|R \cap A|}(\a).\]
    Summing over $A$ and using \cref{Prop:NewCombFormula1},
    we get our theorem.
\end{proof}

\subsection{Nonnegativity in  \texorpdfstring{$\a$}{a} falling factorial basis}
\label{SubsectKkFFPos}

\begin{proof}[Proof of \cref{ThmOnePart}]
We start from \cref{thm:NewCombFormula2}:
\[
\tfrac{1}{k!} \shJack_{(k)}(\lambda)=
\Ko_{(k)}^{(\alpha)}(\lambda) =  \sum_{A \in \CD(\la,k)} w(A),\]
where $\CD(\la,k)$ is the set of column-distinct subsets $A$ of $\la$ of size $k$
and $w(A) \coloneqq \prod_R P_{|R \cap A|}(\a)$, the product running over rows of $\la$.
We will explain how to associate to sets $A$ a \emph{skeleton} $\hat{A}$,
and split the sum above according to this skeleton.

Let $\la=\rvec^\pvec$ be given by its multirectangular coordinates; see \cref{FigDiagMulti}.
Each box in $\la$ belongs to some unique rectangle $p_i \times r_j$. 
We label the corresponding row and column by $i$ and $j$, respectively.
Then we remove rows and columns which contain no boxes of $A$.
For example, \cref{eq:skeletonexample} shows a set $A$ with $5$ boxes
in $\la=\rvec^\pvec$
for $\pvec=(1,2,2,3)$, $\qvec = (4,2,3,2)$
and the construction of its skeleton $\hat{A}$.

\begin{equation}
\ytableausetup{mathmode,boxsize=.7em}
A=
\begin{yt}
\none[\;]  & \none[{\scriptstyle 4}] & \none[{\scriptstyle 4}] & \none[{\scriptstyle 4}] & \none[{\scriptstyle 4}] &
\none[{\scriptstyle 3}] & \none[{\scriptstyle 3}] & \none[{\scriptstyle 2}] &\none[{\scriptstyle 2}] & \none[{\scriptstyle 2}] &
\none[{\scriptstyle 1}] & \none[{\scriptstyle 1}] \\
    \none[{\scriptstyle 1}] &\; &\; & \;  & \;  & \; & \; & \; & \; & \ast & \ast & \; \\
    \none[{\scriptstyle 2}] &\; &\; & \;  & \;  & \; & \; & \; & \; & \; \\
    \none[{\scriptstyle 2}] &\ast &\; & \;  &\;  & \; & \; & \; & \; & \; \\
    \none[{\scriptstyle 3}] &\; &\; & \;  & \;  & \; & \; \\
    \none[{\scriptstyle 3}] &\; &\; & \;  & \;  & \; & \; \\
    \none[{\scriptstyle 4}] &\; &\; & \;  & \; \\
    \none[{\scriptstyle 4}] &\; &\; & \;  & \; \\
    \none[{\scriptstyle 4}] &\; &\ast& \;  & \ast
\end{yt}
\to\
\hat{A}=
\begin{yt}
    \none[\;]  & \none[{\scriptstyle 4}] & \none[{\scriptstyle 4}] & \none[{\scriptstyle 4}]
    & \none[{\scriptstyle 2}] &  \none[{\scriptstyle 1}] \\
    \none[{\scriptstyle 1}] &\; &\; & \; & \ast & \ast  \\
    \none[{\scriptstyle 2}] &\ast & \; & \; \\
 \none[{\scriptstyle 4}] &\; & \ast & \ast
\end{yt}
\label{eq:skeletonexample}
\end{equation}
Fix a skeleton $\hat{A_0}$.
We define $a_i$ and $b_i$ as the number of rows and columns, respectively, labeled $i$.
In the example above, $\avec=(1,1,0,1)$ and $\bvec=(1,1,0,3)$.

Note that if the size $k$ of $|A|$ and the number of entries in $\pvec$ and $\rvec$ is fixed,
there is only a finite number of such possible skeletons.
Moreover, $A$ being column-distinct is equivalent to its skeleton being column-distinct.
Also note that if two sets $A$ have the same skeleton, they also have the same weight $w(A)$.
We therefore have:
\begin{equation}
\sum_{ \substack{ A \in \CD(\rvec^\pvec, k), \\ \hat{A}=\hat{A_0} } } w(A) = w(A_0) \cdot | \{ A \in \CD(\rvec^\pvec, k), \hat{A}=\hat{A_0} \}|
\label{Eq:SumTech}
\end{equation}

It suffices to show that for each skeleton $\hat{A_0}$,
the quantity \eqref{Eq:SumTech}
has non-negative coefficients in the falling factorial basis.
Since $w(A_0)$ is a nonnegative polynomial in $\a$,
we focus on the number of sets $A$ with a given skeleton $\hat{A_0}$.
\medskip

To construct such a set $A$ of shape $\la=\rvec^\pvec$,
we need to choose how to distribute the rows and columns.
Consider $\avec$ and $\bvec$ in the skeleton $\hat{A_0}$ as above.
For each $i$, we need to choose rows of $\la$
corresponding to the $a_i$ rows labeled $i$ in $\hat{A}=\hat{A_0}$.
Since $\la$ has $p_i$ rows labeled $i$,
this can be done in $\binom{p_i}{a_i}$ ways.
A similar reasoning holds for selecting the columns, so there are in total
\[
 \prod_{i} \binom{p_i}{a_i} \times \prod_{j} \binom{r_j}{b_j}
\]
subsets $A$ of $\la$ with shape $\rvec^\pvec$ that have $\hat{A_0}$ as skeleton.
This expression is non-negative in the falling factorial basis.
\end{proof}

\begin{remark}
    A similar proof, starting directly from \cref{Prop:Kok_FromKS},
    would not be possible. 
    Indeed, Condition \ref{item:forbidden} in \cref{def:KSHT},
    as well as the definition of critical boxes,
    do not depend only on the skeleton as they involve
    the notion of \emph{adjacent columns}.
    \label{rk:KSTableuaux_NotAdequate}
\end{remark}

\section{Computer experiments}
\label{SectComputer}

We now briefly describe our computer methods to verify \cref{ConjMain} for small $\mu$.
To our knowledge, there is no straightforward way to obtain $\jackJ^{\star,(\alpha)}_\mu(\rvec^\pvec)$
in terms of multi-rectangular coordinates. 
Our method is based on data available from Lasalle's web page\footnote{\url{http://igm.univ-mlv.fr/~lassalle/free.html}}:
Lassalle gives a table of expressions of
 $\Ch_{\mu}^{(\a)}$ in terms of \emph{free cumulants},
 for small partitions $\mu$.
A similar table for shifted Jack is computed using \cref{PropShJackOnCha}.
Finally, free cumulants are expressed in terms of multi-rectangular coordinates,
using the recursive equations given in \cite{RattanStanleyTopDegree}
(we have slightly adapted these equations 
since Rattan considered Stanley's version of multirectangular coordinates,
while we use the more symmetric one, but this does not create any difficulty).
Then it is straightforward to expand this in the falling factorial basis and check for positivity.


It would be interesting to find an algorithm to directly compute
$\jackJ^{\star,(\alpha)}_\mu(\rvec^\pvec)$
in terms of multi-rectangular coordinates. 

\section*{Acknowledgements}
We thank Maciej Dołęga and Piotr \'Sniady for numerous discussions on topics related to this work.
We thank Michel Lassalle for making his numerical data available on his web page.

VF is partially supported by the grant SNF-149461 ``Dual combinatorics of Jack polynomials''.
  
PA is partially supported by the grant  ``Knut and Alice Wallenberg Foundation'' (2013.03.07).

\bibliographystyle{alpha}
\bibliography{./biblio}

\end{document}